\documentclass[11pt,a4paper]{amsart}
\usepackage{amsmath}
\usepackage{amsfonts}
\usepackage{amssymb}
\usepackage{amsthm}
\usepackage{tikz-cd}
\usepackage[utf8]{inputenc}
\usepackage{csquotes}
\usepackage[unicode=true, pdfusetitle, colorlinks=true, allcolors=blue, bookmarks=false]{hyperref}
\usepackage[english]{babel}
\usepackage{subcaption}
\usepackage[citestyle=numeric-comp,sorting=nyt,sortcites, backend=bibtex]{biblatex}

\bibliography{article.bib}
\AtEveryBibitem{
  \clearfield{doi}
  \clearfield{url}
  \clearfield{month}
  \clearfield{day}
  \clearfield{issn}
  \clearfield{pages}
  \clearfield{publisher}
  \clearfield{isbn}
  \clearfield{pagetotal}
}
\DeclareFieldFormat*{urldate}{}
\DeclareFieldFormat*{title}{#1}
\renewbibmacro{in:}{}
\DeclareFieldFormat*{volume}{vol. #1}
\DeclareFieldFormat*{number}{no. #1}
\renewbibmacro*{volume+number+eid}{%
  \printfield{volume}%
  \setunit*{\addcomma\space}
  \printfield{number}%
  \setunit{\addcomma\space}%
  \printfield{eid}
}
\DeclareFieldFormat[misc]{date}{(#1)}
\newtheorem{theorem}{Theorem}[section]

\newtheorem{corollary}[theorem]{Corollary}
\newtheorem{lemma}[theorem]{Lemma}
\newtheorem{definition}[theorem]{Definition}
\newtheorem{remark}[theorem]{Remark}

\newcommand{\inclusion}{\hookrightarrow}

\newcommand{\nonsep}{\mathcal{N}}

\newcommand{\lis}{locally injective simplicial }
\DeclareMathOperator{\mcgo}{Mod}
\DeclareMathOperator{\mcg}{Mod^\pm}
\DeclareMathOperator{\aut}{Aut}

\author{Rodrigo De Pool}
\address{Instituto de Ciencias Matem\'aticas (ICMAT). Madrid, Spain}
\email{rodrigo.depool@icmat.es}

\title{Finite rigid sets of the non-separating curve complex}

\begin{document}

\begin{abstract}
  We prove that the non-separating curve complex of every surface of  finite type and genus at least three admits an exhaustion by finite rigid sets.
\end{abstract}

\maketitle

%%%%%%%%%%%%%%%%%%%%%%
%%% INTRODUCTION
%%%%%%%%%%%%%%%%%%%%%%
\section{Introduction}

Let $S$ be a connected, orientable, finite-type surface. The \emph{curve complex} is the simplicial complex $\mathcal{C}(S)$  whose $k$-simplices are sets of to $k+1$ distinct isotopy classes of essential simple curves on $S$ that are pairwise disjoint.

The extended mapping class group, denoted $\mcg(S)$, acts naturally on the set of curves up to isotopy on $S$. This  action preserves  disjointness of curves, and therefore  extends to an action on the complex $\mathcal{C}(S)$. Via this action, the curve complex  works as a combinatorial model to study properties of $\mcg(S)$. For instance, a celebrated theorem of Ivanov in \cite{ivanov_automorphism_1997} asserts that for sufficiently complex surfaces the group $\mcg(S)$ is isomorphic to the group of simplicial automorphisms of the curve complex, a result  commonly known as \emph{simplicial rigidity.} In turn, this result is a key ingredient in establishing the isomorphism $\aut(\mcg(S))\cong \mcg(S)$.

The curve complex, and its applications to the mapping class group, has motivated the study of similar complexes associated to surfaces. For example, simplicial rigidity has been established for: the arc complex \cite{irmak_injective_2010}; the non-separating curve complex  \cite{irmak_complexes_2006}; the separating curve complex \cite{brendle_commensurations_2004}, \cite{kida_automorphisms_2011}; the Hatcher-Thurston complex \cite{irmak_automorphisms_2007}; and the pants graph \cite{margalit_automorphisms_2004} (see \cite{papadopoulos_simplicial_2012} for a survey on complexes associated to surfaces).

Another notion of rigidity which has been of recent interest is that of \emph{finite rigidity}: the simplicial complex $\mathcal{C}(S)$ is said to be finitely rigid if there exists a finite subcomplex $X$ such that any locally injective simplicial map \[\phi:X\to \mathcal{C}(S)\]is induced by a \emph{unique} mapping class, that is, there exists a unique  $h\in \mcg(S)$ such that the simplicial action $h:\mathcal{C}(S)\to \mathcal{C}(S)$ satisfies $h|_X = \phi$. Such $X$ is called a \emph{finite rigid set} of $\mathcal{C}(S)$ with \emph{trivial pointwise stabilizer}.

The finite rigidity of the curve complex was proven by Aramayona and Leininger in \cite{aramayona_finite_2013}, thus answering a question by Lars Louder. Furthermore, they constructed in \cite{aramayona_exhausting_2016} an exhaustion of $\mathcal{C}(S)$ by finite rigid sets with trivial pointwise stabilizers, thus recovering Ivanov's result \cite{ivanov_automorphism_1997} on the simplicial rigidity of $\mathcal{C}(S)$.

Following the result of Aramayona and Leininger,  finite rigidity has been proven for other complexes:  Shinkle proved it for the arc complex \cite{shinkle_finite_2020}   and the flip graph \cite{shinkle_finite_2022}; Hernández, Leininger and Maungchang proved a slightly different notion for the pants graph \cite{hernandezhernandez_finite_2021}, \cite{maungchang_finite_2018}; and Huang and Tshishiku proved a weaker notion for the separating curve complex \cite{huang_finite_2021}.

The main goal of this article is to prove the finite rigidity of the \emph{non-separating curve complex}, $\nonsep(S)$, which is the subcomplex of $\mathcal{C}(S)$ spanned by the non-separating curves.  To prove the finite rigidity of $\nonsep(S)$, one would like to  restrict the ﬁnite rigid set of $\mathcal{C}(S)$ in \cite{aramayona_finite_2013} to $\nonsep(S)$, however it is not clear why this restriction would yield a finite rigid set in $\nonsep(S)$ as their proof uses separating curves in a  fundamental way. Below, we construct a different subcomplex of $\nonsep(S)$ and prove its rigidity.

Our main result is compiled in the next theorem.

\begin{theorem}\label{thm:fr}
  Let $S$  be a connected, orientable, finite-type surface of genus $g\geq 3$. There exists a finite simplicial complex $X\subset \nonsep(S)$ such that any locally injective simplicial map \[\phi:X \to \nonsep(S)\] is induced by a unique  $h\in \mcg(S)$.
\end{theorem}

Our second result produces an exhaustion of the non-separating curve complex by finite rigid sets:

\begin{theorem}\label{thm:exh}
  Let $S$  be an orientable finite-type surface of genus $g\geq 3$. There exist subcomplexes $X_1\subset X_2 \subset \dots \subset \nonsep(S)$ such that \[\bigcup_{i=1}^\infty X_i = \nonsep(S)\] and each $X_i$ is a finite rigid  set with trivial pointwise stabilizer. 
\end{theorem}

From Theorem \ref{thm:exh} we can recover the simplicial rigidity of $\nonsep(S)$ (see \cite[Theorem 1.1]{irmak_edge-preserving_2020}):

\begin{corollary}
  Let $S$ be a connected, orientable, finite-type surface of genus $g\geq 3$. Any locally injective simplicial map $\phi: \nonsep(S)\to \nonsep(S)$ is induced by a unique $h\in \mcg(S)$. In particular, this yields an isomorphism between $\mcg(S)$ and  the group of simplicial automorphisms of $\nonsep(S)$.
\end{corollary}

\subsection*{Plan of the paper} In Section \ref{sec:preliminaries} we introduce some basic definitions that will be required. In Section \ref{sec:frs} we introduce the notion of finite rigid sets. Sections \ref{sec:fr_no_puncs} and \ref{sec:exh_no_puncs} deal with the proofs of Theorems \ref{thm:fr} and \ref{thm:exh} for closed surfaces. Lastly, sections \ref{sec:fr_puncs} and \ref{sec:exh_puncs} present the proofs of theorems \ref{thm:fr} and \ref{thm:exh} for punctured surfaces.

\subsection*{Acknowledgements}
The author would like to thank his supervisor Javier Aramayona for suggesting the problem and for the guidance provided. 

The author acknowledges financial support from grant CEX2019-000904-S funded by MCIN/AEI/ 10.13039/501100011033.

%%%%%%%%%%%%%%%%%%%%%%
%%% PRELIMINARIES
%%%%%%%%%%%%%%%%%%%%%%
\section{Preliminaries}\label{sec:preliminaries}

  Let $S$ be connected, orientable surface without boundary. We will further assume that $S$ has finite type, i.e, $\pi_1(S)$ is finitely generated. As such, $S$ is homeomorphic to $S_{g,n}$, the result of removing $n$ points from a genus $g$ surface. We refer to the removed points as \emph{punctures}. If $S$ has no punctures, we will say $S$ is \emph{closed}. Otherwise, we will refer to $S$ as a \emph{punctured} surface.

  Before fervently jumping into the proofs, we warn the reader that the classification of surfaces, the change of coordinates principle and the Alexander method will be frequently used in proofs, sometimes without mention. For these and other fundamental results on mapping class groups, we refer the reader to  \cite{farb_primer_2012}.

  \subsection{Curves}
  By a curve $c$ in $S$ we will mean the isotopy class of an unoriented simple closed curve that does not bound a disk or a punctured disk. Throughout the article, we will  make no distinction  between  curves and their  representatives. We will say $c$ is \emph{non separating} if any  representative  $\gamma$  of $c$ has connected complement in $S$.

 The \emph{intersection number} $i(a,b)$  between two curves $a$ and $b$ is the minimum intersection number between representatives of $a$ and $b$. If $i(a,b)=0$, we will say the curves $a$ and $b$ are \emph{disjoint}. Given two representatives $\alpha \in a$ and $\beta \in b$, we say they are in \emph{minimal position} if $i(a,b)=|\alpha \cap \beta|$. A fact that will be often used is that for any set of curves we may pick a single representative for each curve, such that the representatives are pairwise in minimal position (see \cite[Chapter 1.2]{farb_primer_2012}).

  Given a set of curves $\{c_1,\dots, c_k\}$, consider representatives $\gamma_i \in c_i$ pairwise  in minimal position. We will denote a regular neighborhood of  $\bigcup \gamma_i$ by $N(\bigcup \gamma_i)$. The set of curves in the boundary of $N(\bigcup \gamma_i)$ will be denoted  
   \[\partial(c_1, \dots, c_k).\]
  We emphasize that implicit in the definition of $b\in \partial(c_1, \dots, c_k)$, is that $b$ is an isotopy class of a simple closed curve which does not bound a disk or a punctured disk.

  \subsection{Non-separating curve complex}
  The \emph{non-separating curve complex}  $\nonsep(S)$ is the  simplicial complex whose  $k$-simplices are sets of $k+1$ isotopy classes of pairwise disjoint curves.

  Note that we can endow $\nonsep(S)$ with a metric by declaring each $k$-simplex to have the standard euclidean metric and considering the resulting path metric on $\nonsep(S)$.

  \subsubsection{Pants decompositions}

  The dimension of $\nonsep(S_{g,n})$ is $3g-3+n$ and the vertex set of a top-dimensional simplex in  $\nonsep(S_{g,n})$ is called a \emph{non-separating pants decomposition} of $S_{g,n}$. If $P=\{c_1,\dots, c_k\}$ is a pants decomposition of $S$, then $S\setminus \bigcup c_i$  is a union of  \emph{pairs of pants} (i.e, a union of subsurfaces homeomorphic to $S_{0,3}$). 

  Let $P=\{c_1,\dots, c_k\}$ be a pants decomposition of the surface $S$. Two curves $c_i,c_j \in P$ are said to be \emph{adjacent rel to P} if there exists a curve $c_k \in P$ such that $c_i, c_j, c_k$ bound a pair of pants in $S$; see Figure \ref{fig:pant_decomp_nopuncs}  for an example. 

  We record the following observation for future use. 
  \begin{remark}\label{rmk:at_least_adj_curves}
    Let $P$ be a non-separating pants decomposition of $S_{g,n}$, where $g\geq 3$:

    \begin{itemize}
    \item  If $n\leq 1$, then every curve in $P$ is adjacent rel to $P$ to at least three other curves.
    \item If $n>1$, then every curve is adjacent to at least two other curves.
  \end{itemize}

  \end{remark}

  Consider $A \subset P$, where $P$ is a pants decomposition of  $S$. We say that a set of curves $\tilde{A}$ \emph{substitutes $A$ in $P$} if \[ (\tilde{A} \cup P) \setminus A \]  is a pants decomposition. In words, we say $\tilde{A}$ substitutes $A$ in $P$ if both sets have no curves in common and we can replace the curves in $A$ by the curves in $\tilde{A}$ and still get a pants decomposition.

  \section{Finite rigid sets}\label{sec:frs}

  For a simplicial subcomplex $X \subset \nonsep(S)$, a map $\phi: X \to \nonsep(S)$ is  said to be a \emph{locally injective simplicial map}  if $\phi$ is  simplicial and injective on the star of each vertex. A first elementary observation is the following:

  \begin{lemma}\label{lemma:elementary_pants_to_pants}
    Let $\phi:X\to \nonsep(S)$ be a locally injective simplicial map. If $P\subset X$ is a pants decomposition, then $\phi(P)$ is a pants decomposition.
  \end{lemma}
  \begin{proof}
    Take a vertex $p\in P$. Since $\phi$ is injective in the star of $p$ and $P$ is a simplex, then $\phi$ is injective on $P$. Thus, $\phi(P)$ is a maximal-dimensional simplex, i.e, $\phi(P)$ is a pants decomposition. 
  \end{proof}

  As mentioned in the introduction, the main goal of this article is to construct a finite subcomplex $X\subset \nonsep(S)$ with the following properties:

  \begin{definition}[Finite rigid set]
    A \emph{finite rigid set} $X$ of $\nonsep(S)$ is a finite subcomplex such that any \lis map $\phi: X \to \nonsep(S)$ is induced by a mapping class, i.e, there exists $h\in \mcg(S)$ with $h|_X = \phi$. 

    In addition, if $h$ is unique, we say $X$ has \emph{trivial pointwise stabilizer}.
  \end{definition}

  Observe that a subcomplex $X\subset \nonsep(S)$ has trivial pointwise stabilizer if and only if the inclusion $X \inclusion \nonsep(S)$ is induced uniquely by the identity $1\in \mcg(S)$, hence the name. 

  \begin{remark}
    By the change of coordinates principle (see \cite[Chapter 1.3]{farb_primer_2012}), every vertex $\{v\} \subset \nonsep(S)$ is a finite rigid set. However, the stabilizer of $\{v\}$ is not trivial.
  \end{remark}
  \begin{remark}
    If $X$ is a finite rigid set and $X\subset Y$, then $Y$ may not be a finite rigid set. For example:
     
    Consider two disjoint curves $v_1,v_2$ such that $S\setminus \bigcup v_i$ is connected, and two disjoint curves $v_1',v_2'$ such that $S\setminus \bigcup v_i'$ is disconnected. Now, take $X=\{v_1\}$, $Y=\{v_1,v_2\}$ and the \lis map $\phi(v_i)=v_i'$. Clearly, $\phi$ is not induced by a mapping class and so $Y$ is not a finite rigid set of $\nonsep(S)$. Note that $X$ is a finite rigid set of $\nonsep(S)$ by the remark above.
  \end{remark}

  Following Aramayona and Leininger in \cite{aramayona_finite_2013}, we will say that a subcomplex $X\subset \nonsep(S)$  \emph{detects the intersection} of two curves $a,b \in X$, if every \lis map $\phi: X \to \nonsep(S)$ satisfies 
\[i(a,b)\neq 0 \Leftrightarrow  i(\phi(a),\phi(b)) \neq 0.\]

%%%%%%%%%%%%%%%%%%%%%%
%%% FR: COMPACT SURFACES
%%%%%%%%%%%%%%%%%%%%%%
\section{Finite rigid sets for closed surfaces}\label{sec:fr_no_puncs}
In this section we construct  finite rigid sets for closed surfaces and prove their rigidity. This will establish  Theorem \ref{thm:fr} for closed surfaces. 

\subsection{Constructing the finite rigid set}

Let $S$ be a closed surface of genus $g\geq 3$. We will start by defining the curves in the finite rigid set. The reader should keep figures \ref{fig:pant_decomp_nopuncs}, \ref{fig:circular_nopuncs}, \ref{fig:ctilde}, \ref{fig:upper}, \ref{fig:el_er} and \ref{fig:ed} in mind throughout the section.

Fix a set $\{p_1,\,c_1,\,\dots,\, p_g, \, c_g, \,p_{g+1}\}$ of non-separating curves such that $i(c_i,p_i)=i(c_i,p_{i+1})=1$ and the rest of the curves are pairwise disjoint  (see figures \ref{fig:pant_decomp_nopuncs} and \ref{fig:circular_nopuncs}). Such a set of curves is unique up to homeomorphism. Let $c_{g+1}$ be a curve such that  $i(p_1,\,c_{g+1})=i(p_{g+1},\, c_{g+1})=1$ and is disjoint from every other curve in the set above. We define \[C=\{c_1,\dots,c_{g+1}\}.\] 

Notice that $S\setminus \bigcup p_i$ has two connected components $(S\setminus \bigcup p_i)^+$ and  $(S\setminus \bigcup p_i)^-$, we will call $(S\setminus \bigcup p_i)^+$ the \emph{ top component} and $(S\setminus \bigcup p_i)^-$ the \emph{bottom component}. In the same fashion, $S\setminus \bigcup c_i$ has two connected components $(S\setminus \bigcup c_i)^+$ and $(S\setminus \bigcup c_i)^-$, we will call $(S\setminus \bigcup c_i)^+$ the \emph{front component} and   $(S\setminus \bigcup c_i)^-$ the \emph{back component}. 

For each $k=2,\dots,g-1$, the set $\partial(p_1,c_1,\dots,p_k,c_k)$  consists of two curves: one of them in $(S\setminus \bigcup p_i)^+$ and the other one  in $(S\setminus \bigcup p_i)^-$. We will call $p_k^+$ the curve  of  $\partial(p_1,c_1,\dots,p_k,c_k)$  contained in  $(S\setminus \bigcup p_i)^+$ and by  $p_k^-$ the curve of  $\partial(p_1,c_1,\dots,p_k,c_k)$  in  $(S\setminus \bigcup p_i)^-$. We set 
\[P=\{p_1\dots,p_{g+1}\}\cup\{p_2^+,p_2^-, \dots, p_{g-1}^+,p_{g-1}^-\}.\]Notice that $P$ is a pants decomposition  (see Figure \ref{fig:pant_decomp_nopuncs}).

For each $k=2,\dots,g-1$, the set $\partial(p_{k-1},c_k,p_k)$ has two curves, one in $(S\setminus \bigcup p_i)^+$ and the other one in $(S\setminus \bigcup p_i)^-$. We will denote by $u_k$ the curve in $(S\setminus \bigcup p_i)^+$  and by $d_k$ the curve in  $(S\setminus \bigcup p_i)^-$ (see Figure \ref{fig:ctilde}). We set 
\[
U=\{u_2,\dots,u_{g-1}\}
\]
and 
\[
D=\{d_2,\dots,d_{g-1}\}.
\]

Given $k\in \{2,\dots,g-1\}$ the set $\partial(p_k,c_k, p_k^+ )$ contains two curves and  only one of them is also a curve in $P$. We will denote by $l_k$ the curve in  $\partial(p_k,c_k, p_k^+ )$ not already in $P$ (see Figure \ref{fig:upper}). We set \[L = \{l_2, \dots, l_{g-2}\}.\] Analogously, let  $R=\{r_2, \dots, r_{g-2}\}$ be the set of curves where $r_k$ is the unique curve in  $\partial( {p_{k+1}}^+, c_{k+1}, p_{k+2})$ that is not in $P$ (see Figure \ref{fig:upper}).

The set $\partial (p_2,\,c_2,\,\dots p_{g-1},\,c_{g-1},\,p_g)$ has two curves, one curve in each component of $S\setminus \bigcup p_i$. Let $b$ be the curve contained in the bottom component $(S\setminus \bigcup p_i)^-$. Then, the set $\partial(c_1,\,b_,\,c_g)$ has exactly two curves, one curve contained in $(S\setminus \bigcup c_i)^+$ and the other in $(S\setminus \bigcup c_i)^-$. Denote by $nd$ the curve  in $(S\setminus \bigcup c_i)^+$ (see Figure \ref{fig:ed}). 

Lastly, consider the torus $T_1$ that contains $p_1$ and is bounded by the curves $p_2^+,p_2^-$. Let $nl$  be the unique curve contained in $T_1\setminus (nd \cup c_1)$ distinct from  $c_1$ and $p_2^+$. In the same way, $p_{g-1}^+, p_{g-1}^-$ bound a torus $T_g$  such that $p_g\subset T_g$, let $nr$  be the unique curve in $T_g\setminus (nd\cup c_g)$ distinct from $c_g$ and $p_{g-1}^+$ (see Figure \ref{fig:el_er}). We set \[N=\{nl, \,nr,\,nd\}.\]

\begin{figure}[h]
  \centering
  \begin{subfigure}{.45\linewidth}
    \centering
    \includegraphics[width=0.8\linewidth]{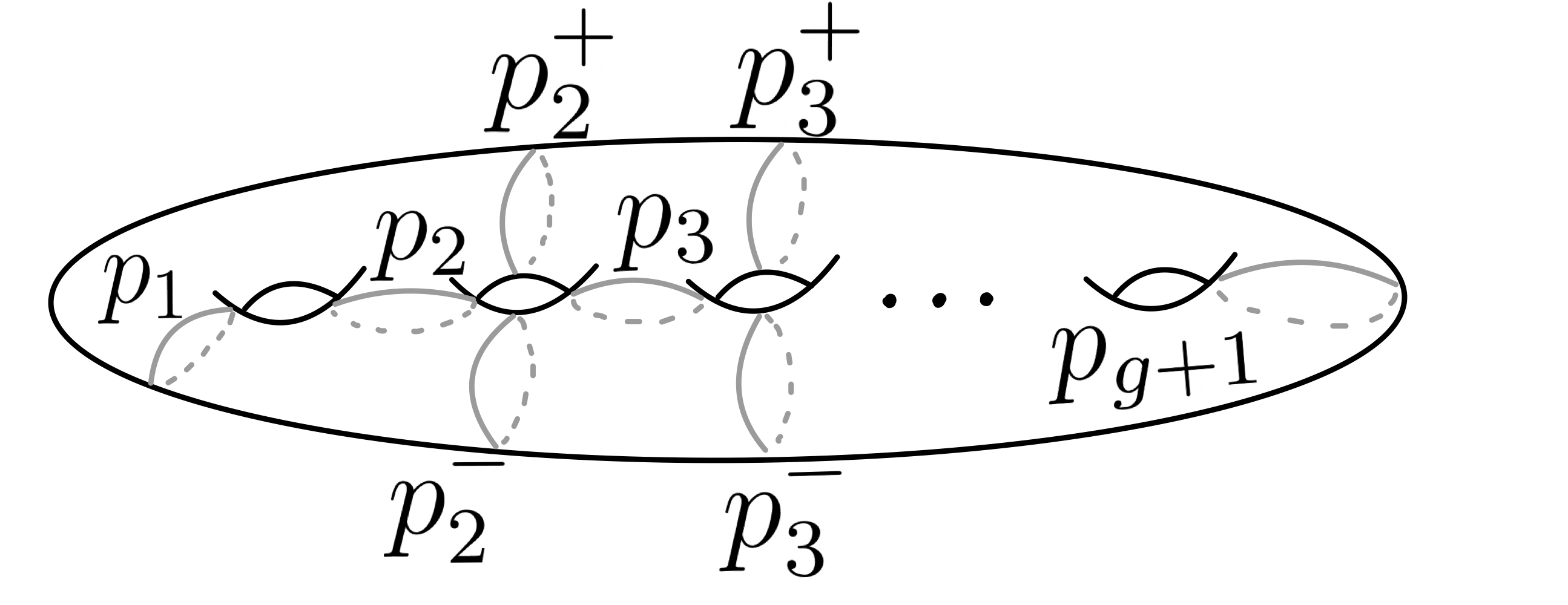}
  \caption{Pants decomposition $P$ of surface $S_g$. As an example, note that $p_1$ and $p_2$ are adjacent rel to $P$; while $p_1$ and $p_3$ are not adjacent rel to $P$.}
  \label{fig:pant_decomp_nopuncs}
  \end{subfigure}
  \begin{subfigure}{.45\linewidth}
    \centering
    \includegraphics[width=0.8\linewidth]{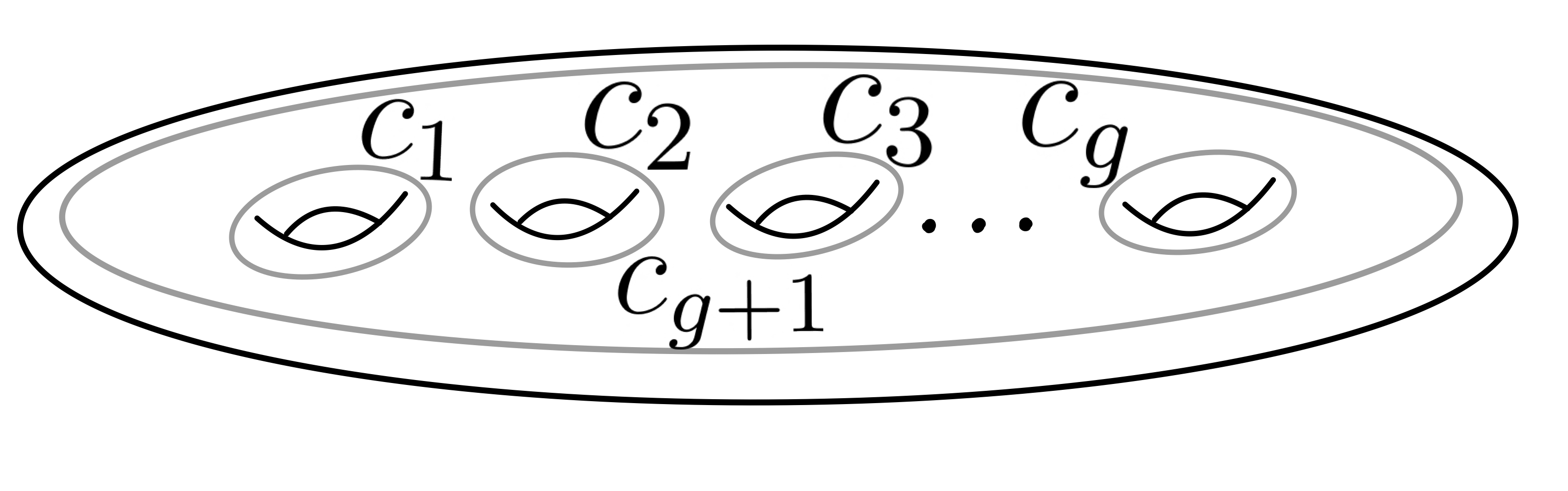}
  \caption{Circular curves $C$.}
  \label{fig:circular_nopuncs}
  \end{subfigure}\par\medskip

  \begin{subfigure}{.45\linewidth}
    \centering
    \includegraphics[width=0.8\linewidth]{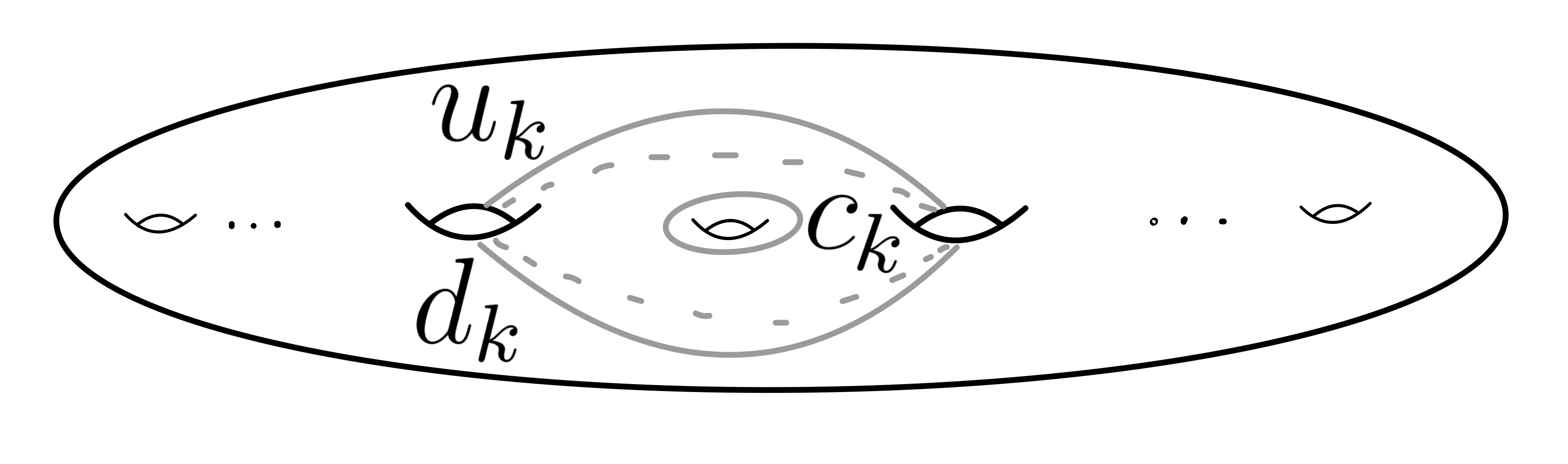}
  \caption{Up and down curves $u_k, d_k$.}
  \label{fig:ctilde}
  \end{subfigure}
  \begin{subfigure}{.45\linewidth}
    \centering
    \includegraphics[width=0.8\linewidth]{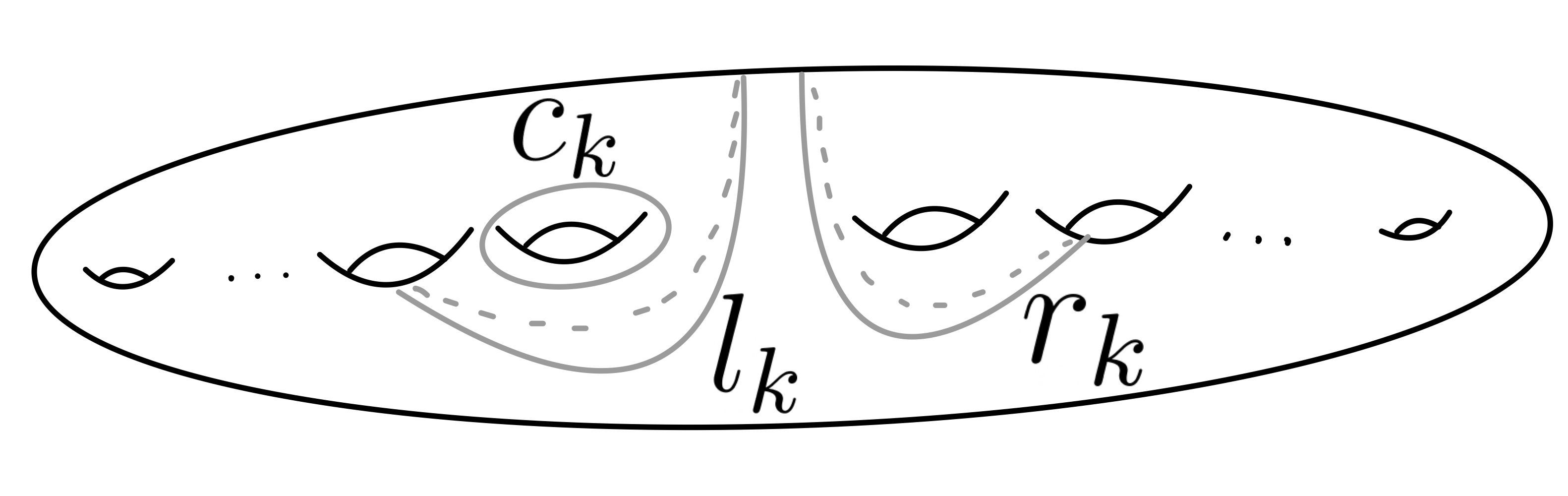}
  \caption{Left and right curves $l_k, r_k$.}
  \label{fig:upper}
  \end{subfigure}\par\medskip

  \begin{subfigure}{.45\linewidth}
    \centering
    \includegraphics[width=0.8\linewidth]{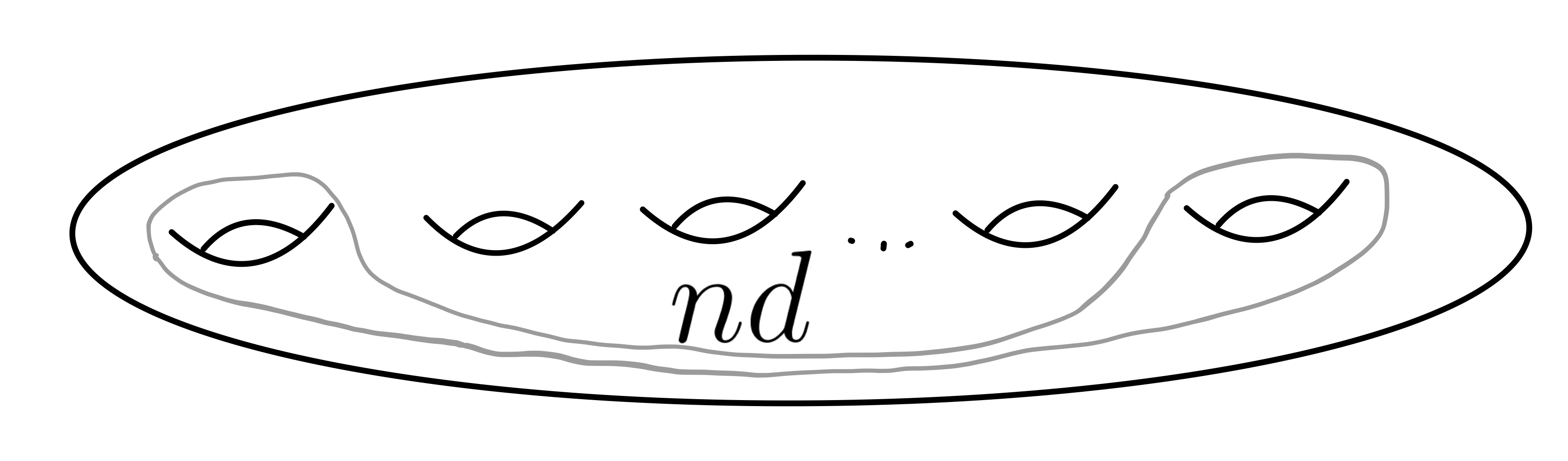}
  \caption{Non-symmetrical down curve $nd$.}
  \label{fig:ed}
  \end{subfigure}
  \begin{subfigure}{.45\linewidth}
    \centering
    \includegraphics[width=0.8\linewidth]{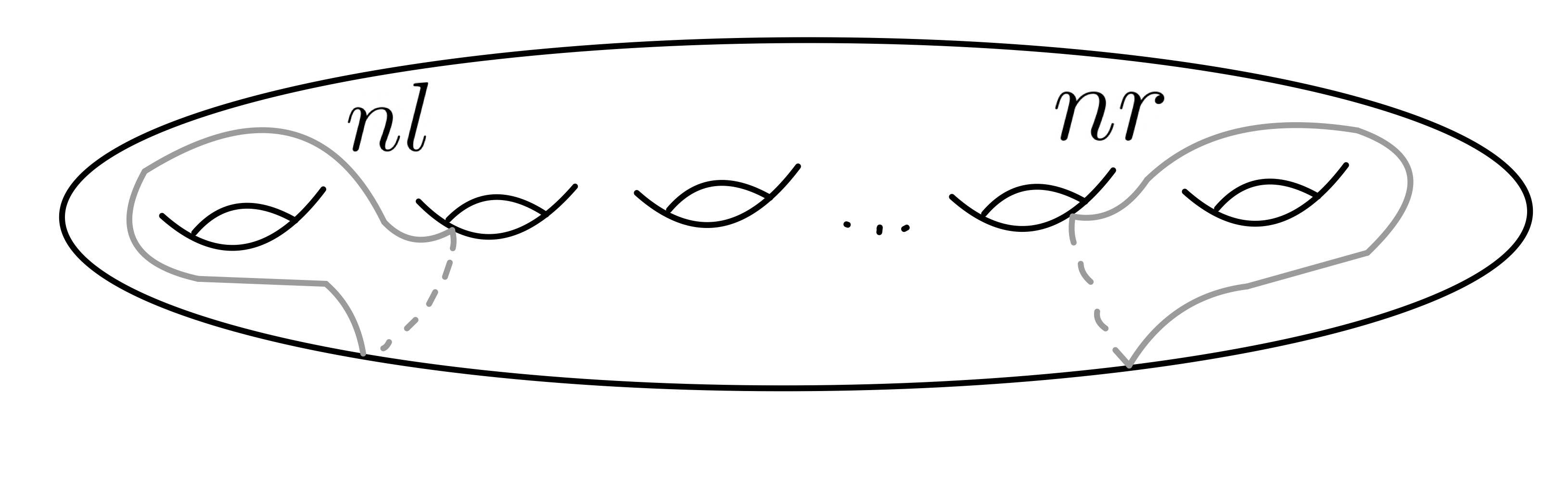}
  \caption{Non-symmetrical left and right curves $nl, \, nr$.}
  \label{fig:el_er}
  \end{subfigure}
  \caption{Curves in $F_R$ for a closed surface.}
\end{figure}

We set $F_R$ to be the   subcomplex of $\nonsep(S)$  spanned by the vertices in 
\[ P \cup C \cup U \cup D \cup L \cup R \cup N.\]

\begin{remark}\label{rmk:dist_two}
  Note that the subcomplex  $F_R$   has diameter two. Thus, any \lis map $\phi: F_R \to \nonsep(S)$ is injective. 
\end{remark}

\subsection{Proving the rigidity of $F_R$} The first step is to check that the \lis map $\phi: F_R \to \nonsep(S)$ preserves the \emph{non} adjacency rel to $P$. To do so, we require the following technical lemma:

\begin{lemma}\label{lemma:non_adj}
  Let $S$ be a finite-type surface,  $F_R \subset \nonsep(S)$ a  subcomplex, $P\subset F_R$ a pants decomposition and $a,b\in P$ two  curves. If there exist subsets $A,B \subset P$ and $\tilde{A}, \tilde{B} \subset F_R$ satisfying the following assertions:
  \begin{itemize}
    \item $a\in A$ and $b\in B$,
    \item $\tilde{A}$ substitutes $A$ in $P$,
    \item $\tilde{B}$ substitutes $B$ in $P$,
    \item $\tilde{A} \cup \tilde{B}$ substitutes $A\cup B$ in $P$, and
    \item $A\cap B=\emptyset$.
  \end{itemize}
  Then $a,b$ are not adjacent rel to $P$.
\end{lemma}
\begin{proof}
  We will proceed by contradiction. Suppose the curves  $a,b$ are adjacent in a pair of pants $Q$ bound by $a,b,c$. Since $A\cap B =\emptyset$ either $c\not \in A$ or $c\not \in B$. Without loss of generality, suppose $c\not \in A$. 
  
  If $\tilde{A}$ is a substitution of $A$ and  $A \cap B = \emptyset$, then there is a curve $\tilde{a}\in \tilde{A}$ such that $i(\tilde{a},  a)\neq 0$, $i(\tilde{a}, b)=0$ and $i(\tilde{a}, c) = 0$. Now, note that since $\tilde{B}$ and $\tilde{A}\cup \tilde{B}$ are  substitutions, there exists  $\tilde{b}\in \tilde{B}$ with $i (\tilde{b}, b )\neq 0$, $i(\tilde{b},a)=0$ and $i(\tilde{b} , \tilde{a}) = 0$. However, is impossible to have arcs  $\tilde{a}\cap Q$ and $\tilde{b}\cap Q$ satisfying the intersections above (see Figure \ref{fig:arcsQ}).
\end{proof}
\begin{figure}
  \begin{center}
  \includegraphics[width=0.30\linewidth]{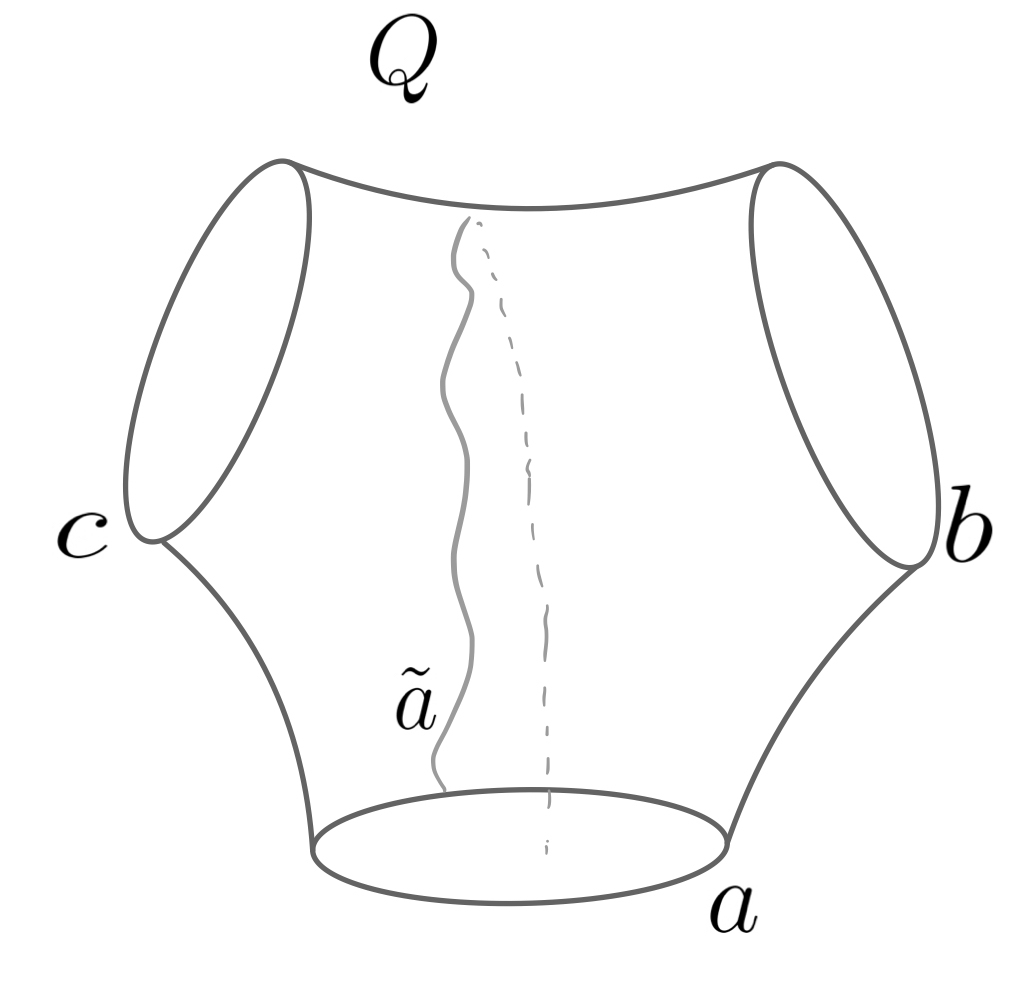}
  \caption{ There can be no arc $\tilde{b}\cap Q$ that intersects $b$ and  does not intersect $\tilde{a}\cup a$.}
  \label{fig:arcsQ}
  \end{center}
\end{figure}

Now, we prove that $\phi$ preserves non-adjacency rel to $P$:

\begin{lemma}\label{lemma:compact_non_adj_rel_p}
  Let $S$ be a  closed surface of genus $g\geq 3$ and  $\phi: F_R \to \nonsep(S)$ a \lis map. If $a,b\in P$ are not adjacent rel to $P$, then $\phi(a),\phi(b)$ are not adjacent rel to $\phi(P)$.
\end{lemma}
\begin{proof}
  Assume that for two curves $a,b \in P$ we have subsets $A,B \subset P$ and $\tilde{A},\tilde{B}\subset F_R$ as in Lemma \ref{lemma:non_adj}. Under these conditions, the lemma ensures that $a$ and $b$ are not adjacent rel to $P$. Moreover, these properties are carried to the image, that is, the curves $\phi(a),\phi(b)\in \phi(P)$ satisfy the conditions of Lemma \ref{lemma:non_adj} for the sets $\phi(A),\phi(B) \subset \phi(P)$ and $\phi(\tilde{A}),\phi(\tilde{B})$. As a consequence, we  deduce that $\phi(a)$ and $\phi(b)$ are not adjacent rel to $\phi(P)$.
  
  By means of the method above, we are only left to find appropriate subsets $A, \tilde{A}, B ,\tilde{B}$  for any non-adjacent curves $a,b\in P$. We will find such subsets for certain $a, b\in P$, as the rest of the cases are similar.

  If $a=p_k$ and $b=p_{k+1}$ for $k\in\{3,\dots, g-3\}$, we can consider $A=\{p_{k-1}^-,p_k\}$ and $\tilde{A}=\{l_{k-1}, d_{k-1}\}$, $B=\{p_{k+1}, p_{k+1}^-\}$ and $\tilde{B}=\{r_k, d_{k+1}\}$. It is straightforward to check that these subsets satisfy conditions of Lemma \ref{lemma:non_adj} and so $\phi(a)$, $\phi(b)$ are not adjacent rel to $\phi(P)$.

  If $a=p_1$. Consider $A = \{p_1,p_2\}$ and $\tilde{A}=\{nl, c_1\}$, then 
  \begin{itemize}
    \item If $b\in \{p_g, p_{g+1}\}$,  take $B = \{p_g,p_{g+1}\}$ and $\tilde{B} = \{nr,c_g\}$. 
    \item If $b \in \{p_k, {p_k^-}\}$ for $k\in \{3, \dots, g-1\}$, consider $B= \{p_k, {p_k^-}\}$ and $\tilde{B} = \{{c_k},r_{k-1}\}$.
    \item If $b=p_k^+$ for $k\in \{3, \dots, g-1\}$, consider $B = \{{p_k^+}\}$ and $\tilde{B}=\{u_k\}$.
  \end{itemize}

\end{proof}

Using the previous result, we prove that $\phi$  preserves adjacency rel to $P$:

\begin{lemma}\label{lemma:compact_adj_rel_p}
  Let $S$ be a  closed surface of genus $g\geq 3$ and   $\phi: F_R \to \nonsep(S)$  a \lis map. If $a,b \in P$ are adjacent rel to $P$, then $\phi(a),\phi(b)$ are adjacent rel to $\phi(P)$.
\end{lemma}
\begin{proof}
  Take $\phi(p_1) \in \phi(P)$. From the non-adjacency rel to $\phi(P)$, it follows that $\phi(p_1)$ has at most three adjacent curves. On the other hand,  Remark \ref{rmk:at_least_adj_curves} implies that  $\phi(p_1)$ has at least three adjacent curves. Thus, we conclude that $\phi(p_1)$ has exactly three adjacent curves, namely $\phi(p_2), \phi(p_2^+), \phi(p_2^-)$. The same argument applies to $\phi(p_2)$, so it is adjacent rel to $\phi(P)$ to exactly three curves, namely $\phi(p_1), \phi(p_2^+), \phi(p_2^-)$.

  We now determine the curves adjacent to $\phi(p_2^+)$ and $\phi(p_2^-)$. First, note that the adjacency rel to $\phi(P)$ of $\phi(p_1)$ and $\phi(p_2)$ implies that  $\phi(p_2^+)$ and $\phi(p_2^-)$ bound a subsurface homeomorphic to $S_{1,2}$. Since both curves $\phi(p_2^+)$ and $\phi(p_2^-)$ are non separating, it follows that both have  four adjacent curves rel to $\phi(P)$. Finally, considering the non adjacency rel to $\phi(P)$, it follows that the curves adjacent to $\phi(p_2^+)$  are   \[\{\phi(p_1), \phi(p_2), \phi(p_3), \phi(p_3^+)\}\] and the curves adjacent to $\phi(p_2^-)$ are  \[ \{\phi(p_1), \phi(p_2), \phi(p_3), \phi(p_3^-)\}.\]
  
  In the same style, we can argue inductively to determine the adjacency of each curve in $\phi(P)$.
\end{proof}

As a corollary we obtain that $\phi$ preserves the topological type of $P$:

\begin{corollary}\label{lemma:fr_top_typ}
  Let $S$ be a  finite-type surface and   $\phi: F_R \to \nonsep(S)$  a map that preserves adjacency rel to $P$. There exists $h\in \mcg(S)$ such that $h|_P=\phi|_P$.
\end{corollary}
\begin{proof}
  We construct a homeomorphism $h$ inductively by gluing abstract pairs of pants.
  
  Consider the pairs of pants $Q_1, \dots, Q_k \subset S$ bound by curves in $P$ and denote $Q'_1, \dots,Q'_k \subset S$ the pairs of pants satisfying $\partial Q'_i = \phi(\partial Q_i)$. Any two pairs of pants are homeomorphic, thus we can consider homeomorphisms $h_i(Q_i)=Q'_i$. Since $\phi$ preserves the adjacency rel to $P$, we can ensure these homeomorphisms agree on the boundary curves. Then, by gluing the maps $h_i$ we obtain a homeomorphism $h$ of the surface such that $h|_P=\phi|_P$.  
\end{proof}

The next three lemmas prove that $F_R$ detects intersection among certain curves. Recall that $F_R$ detects the intersection between $a$ and $b$ if for any locally injective simplicial map $\phi:F_R\to \nonsep(S)$ we have that 
\[i(a,b)\neq 0 \Leftrightarrow  i(\phi(a),\phi(b)) \neq 0.\]

\begin{lemma}\label{lemma:intersections_1}
  The subcomplex $F_R\subset \nonsep(S)$ detects the following intersections for every $k=2,\dots, g-1$:
  \begin{enumerate}
    \item $u_k$ with $p_k^+$ and \label{lema:1_1}
    \item $d_k$ with $p_k^-$. \label{lema:1_2}
  \end{enumerate}
\end{lemma}
\begin{proof}
  Let $\phi:F_R \to \nonsep(S)$ be a locally injective simplicial map. We need to check that $i(\phi(u_k),\phi(p_k^+)) \neq 0$  and $i(\phi(d_k),\phi(p_k^-)) \neq 0$. 

  Seeking a contradiction to case (\ref{lema:1_1}), we assume $i(\phi(u_k),\phi(p_k^+)) = 0$. Since $\phi$ is locally injective, it sends disjoint curves to disjoint curves. Thus, $\phi(u_k)$ is disjoint from every curve in the pants decomposition $\phi(P)$, which implies $\phi(u_k)\in \phi(P)$. However, this contradicts the injectivity of $\phi$ (see Remark $\ref{rmk:dist_two}$). 
  
  To prove the case (\ref{lema:1_2}), the same argument works.
\end{proof}
\begin{remark}
  Notice that from the previous lemma we actually know that $F_R$ detects the intersection of $u_k$ with every curve in $P$. Indeed, $u_k$ is disjoint from any curve in $P\setminus \{p_k^+\}$ and $\phi$ preserves disjointness. In the same way, $F_R$ detects the intersection of $d_k$ with every curve in $P$. 
\end{remark}

\begin{lemma}\label{lemma:aux10}
  The subcomplex $F_R\subset \nonsep(S)$ detects the intersection of $c_k$ with  $p_k^-$ and $p_k^+$, for every $k\in \{2,\,\dots,\, g-1\}$.
\end{lemma}
\begin{proof}
  Let $\phi:F_R \to \nonsep(S)$ be a locally injective simplicial map. By Corollary \ref{lemma:fr_top_typ}, there exists $h\in \mcg(S)$ such that $h\circ \phi$ fixes the pants decomposition $P$. Observe that detecting  intersection is equivalent for $\phi$ and for $h\circ \phi$, since we have  
  \[i(\phi(a),\phi(b))\neq 0 \Leftrightarrow  i(h\circ \phi(a), h\circ \phi(b)) \neq 0.\] And so we can rename $h\circ \phi$ to $\phi$, and prove the statement assuming $\phi$ fixes every $p\in P$.

  With the previous simplification, the proof boils down to check that $\phi(c_k)$ intersects $ \,p_k^+$ and $p_k^-$. In this direction, consider the torus $T_k$ bounded by the curves $p_{k-1}^-,\,p_{k-1}^+,\,p_{k+1}^-$ and  $p_{k+1}^+$ (see Figure \ref{fig:upper_torus}). Further, define the top of the torus as $T_k^+=T_k\cap (S\setminus \bigcup p_i)^+$ and the bottom of the torus as $T_k^-=T_k\cap (S\setminus \bigcup p_i)^-$.

  \begin{figure}[h!]
    \centering
    \includegraphics[width=0.45\linewidth]{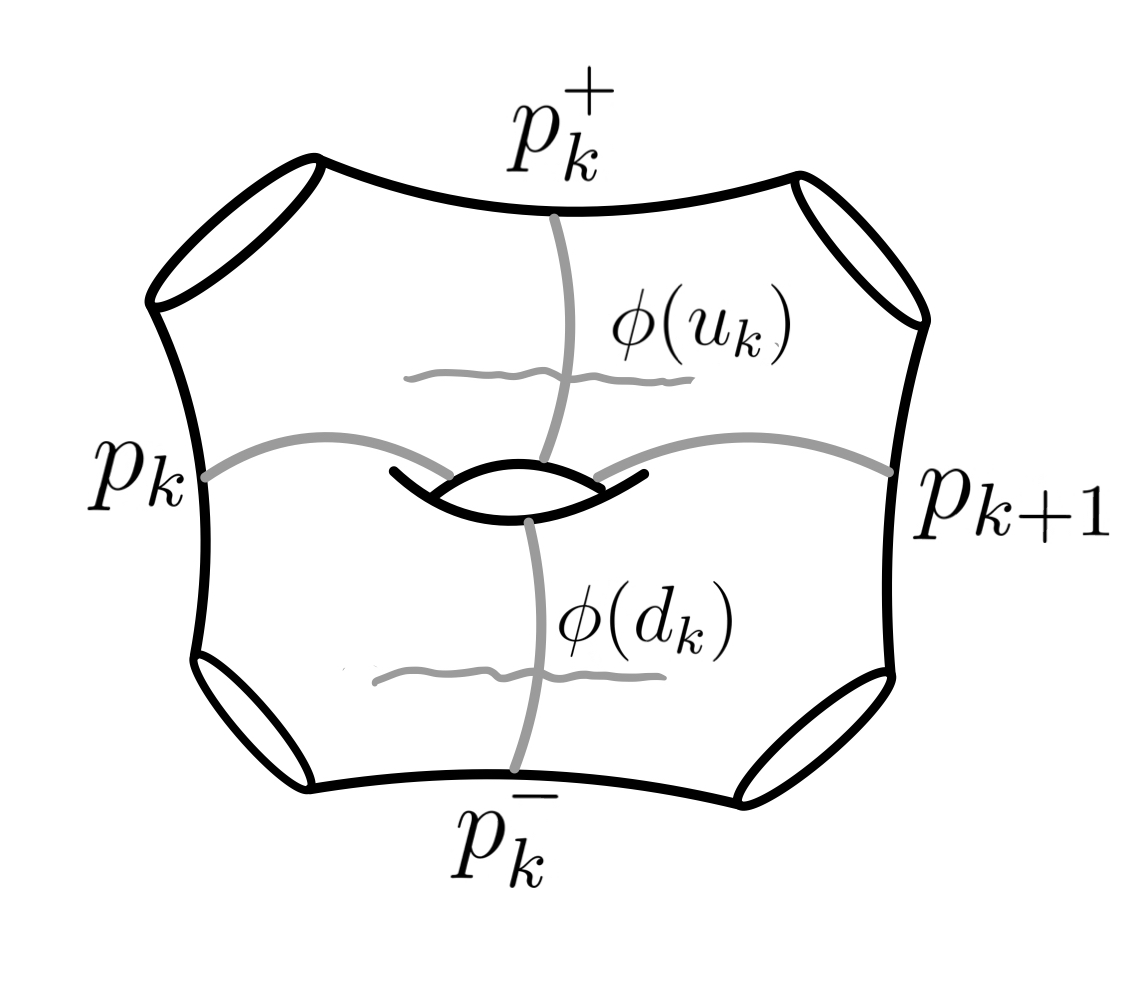}
    \caption{Torus $T_k$.}
    \label{fig:upper_torus}
  \end{figure}

  By Lemma \ref{lemma:intersections_1},  $\phi(u_k)$ is a curve in $T_k^+$ intersecting $p_k^+$  and $\phi(d_k)$ is a curve in $T_k^-$  intersecting $p_k^-$ (see Figure \ref{fig:upper_torus}). Notice $\phi(c_k)$ is a curve in $T_k$ distinct and disjoint from  $\phi(d_k)$ and $\phi(u_k)$. It follows that $\phi(c_k)$ intersects both  $T_k^+$ and $T_k^-$.

  To finish, note  $\phi(c_k)\cap T_k^+$ is disjoint from $\phi(u_k)$, so $\phi(c_k)$ must intersect $p_k^+$. Indeed, an arc disjoint from a curve in a sphere with four boundary components, must intersect every other curve in the sphere. Similarly, $\phi(c_k)\cap T_k^-$ is disjoint from $\phi(d_k)$, so $\phi(c_k)$ must intersect $p_k^-$.
\end{proof}

\begin{lemma}\label{lemma:detect_intersection_closed}
  The subcomplex $F_R\subset \nonsep(S)$ detects the intersection of $c_k\in C$ with every curve in $P$.
\end{lemma}
\begin{proof}
  Let $\phi:F_R \to \nonsep(S)$ be a locally injective simplicial map. As in the previous proof, we may assume that $\phi$ fixes every curve in $P$. 
  
  Now,  we start by proving the cases $2\leq k\leq g-1$. Note that with the simplification above and Lemma \ref{lemma:aux10}, we only need to check that $\phi(c_k)$ intersects $p_k$ and $p_{k+1}$: 

  To prove  $\phi(c_k)$ intersects $p_k$, consider the pair of pants $Q$ bounded by the curves $p_{k-1}^+, \,p_k$ and $p_k^+$. By Lemma \ref{lemma:aux10}, there are disjoint arcs $Q\cap \phi(c_{k-1})$ and $Q\cap \phi(c_k)$ with at least one endpoint in $p_{k-1}^+$ and $p_k^+$, respectively (see Figure \ref{fig:upper_pants}). Using that $\phi(c_{k-1}),\,p_k^+$ are disjoint and $\phi(c_k), \,p_{k-1}^+$ are also disjoint, we conclude that any such arc configuration requires $\phi(c_k)$ to intersect $p_k$. The same argument yields  that $\phi(c_k)$ intersects $p_{k+1}$.

  \begin{figure}[h!]
    \centering
    \includegraphics[width=0.45\linewidth]{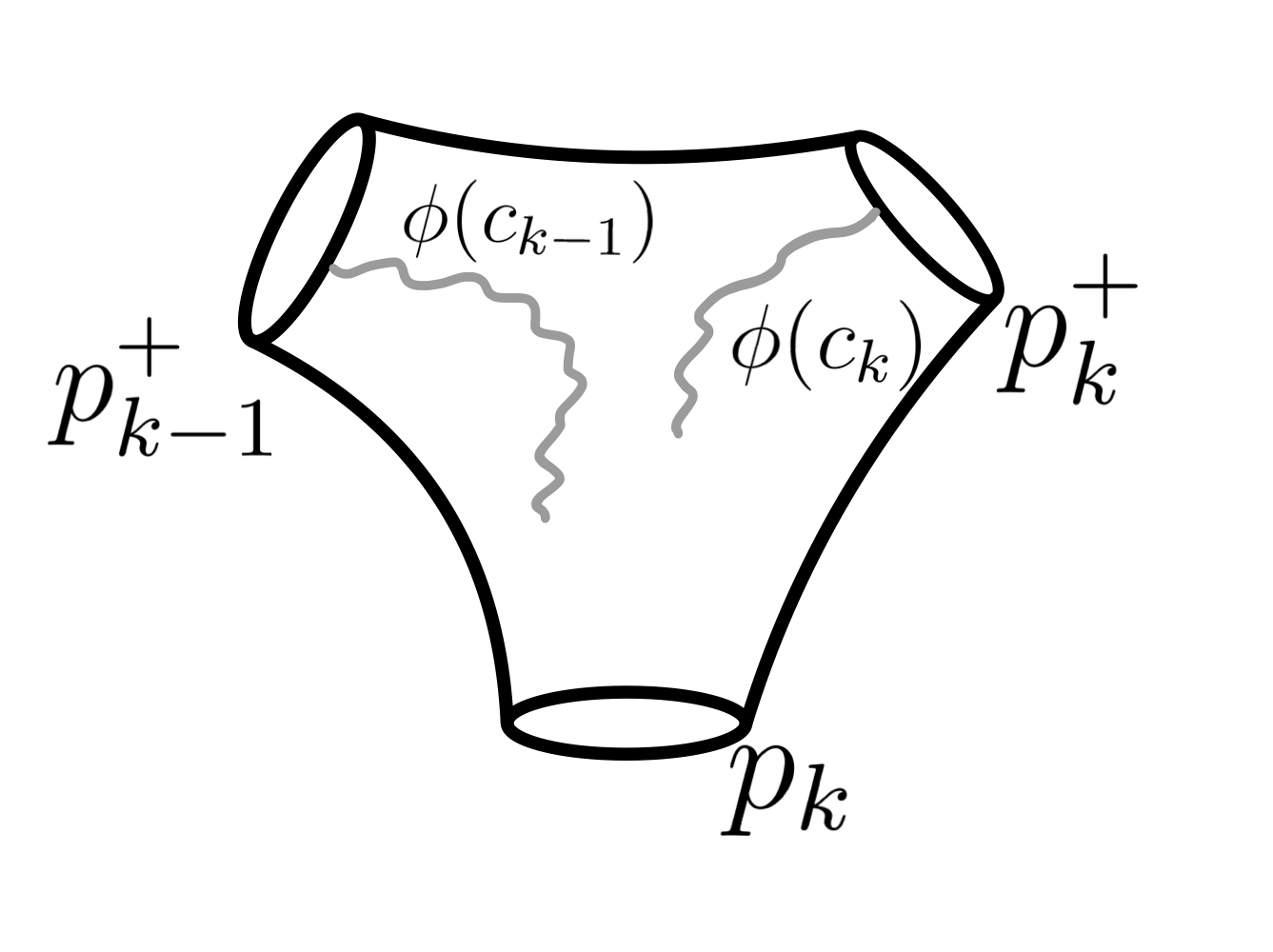}
    \caption{Pair of pants bounded by $p_{k-1}^+, \,p_k$ and $p_k^+$.}
    \label{fig:upper_pants}
  \end{figure}

We are left to prove the cases $k\in\{1,\,g,\,g+1\}$. First, we prove that $\phi(c_1)$ intersects $p_2$. Consider the torus $T_1$ bounded by the curves $p_2^+$ and $p_2^-$, and  denote by $T_1^+$ the pair of pants bounded by $p_1,p_2$ and $p_2^+$ (see Figure \ref{fig:upper_tone}). Note that $\phi(c_1)$ is a curve in $T_1$ distinct from $p_1,\, p_2, \,p_2^-$ and $p_2^+$, it follows that $\phi(c_1)$ intersects $T_1^+$. Thus, we have disjoint arcs $\phi(c_1) \cap T_1^+$ and $\phi(c_2)\cap T_1^+$, the last one having an endpoint in $p_2^+$ (see Figure \ref{fig:upper_tone}). Since $\phi(c_1),\,p_2^+$ are disjoint and $\phi(c_2),\, p_1$ are disjoint, we conclude that $\phi(c_1)$ must intersect $p_2$. Again, the same argument with slight changes yields that $\phi(c_g)$ intersects $p_g$.  

  \begin{figure}[h!]
    \centering
    \includegraphics[width=0.45\linewidth]{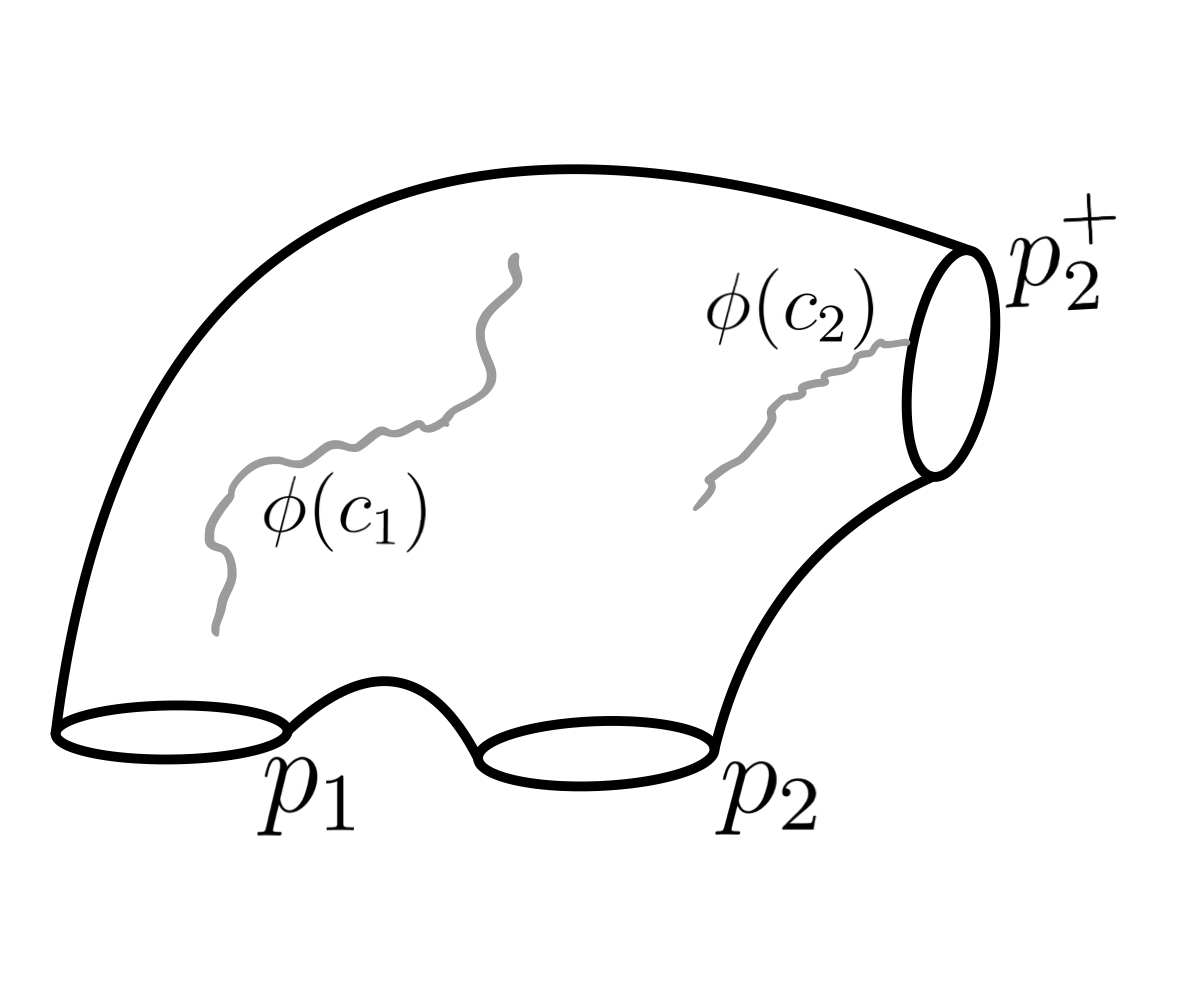}
    \caption{Pair of pants $T_1^+$.}
    \label{fig:upper_tone}
  \end{figure}
  
  Before proving that $\phi(c_1)$ intersects $p_1$ and $\phi(c_g)$ intersects $p_{g+1}$, we need to check that $\phi(c_{g+1})$ intersects both $p_1,\,p_g$ and every $p_k^+, p_k^-$.  Surely, $\phi(c_{g+1})$ intersects at least one of these curves, since otherwise $\phi(c_{g+1})$ would be disjoint and distinct from every curve in the pants decomposition $\phi(P)$. Suppose $\phi(c_{g+1})$ intersects $p_k^+$, and consider the pair of pants $Q$ bound by $p_{k-1}^+,\,p_k$ and $p_k^+$. Using the intersections above, we deduce that there are  disjoint arcs $\phi(c_{k-1})\cap Q$ and $\phi(c_{g+1})\cap Q$, so that $\phi(c_{g+1})$ must also intersect $p_{k-1}^+$. Repeating the argument iteratively, we can detect the intersection of $\phi(c_{g+1})$ with every curve in $P$.

  To finish the proof, we check that $\phi(c_1)$ intersects $p_1$. Consider the disjoint arcs  $a_2=\phi(c_2)\cap T_1^+$ and $a_{g+1}=\phi(c_{g+1})\cap T_1^+$, where $a_2$ has an endpoint in $p_2^+$ and $a_{g+1}$ has an endpoint in $p_1$. These two arcs exist by the above intersections. Moreover, $a_2$ is disjoint from $p_1$, and $a_{g+1}$ is disjoint from $p_2$. Recall that $\phi(c_1)$ intersects $T_1^+$ and, since it is disjoint from both $a_1,a_2$, it follows that $\phi(c_1)$ intersects $p_1$. A similar argument yields that $\phi(c_g)$ intersects $p_{g+1}$.
\end{proof}

So far we have seen that the map $\phi:F_R\to \nonsep(S)$ can be taken to agree with a homeomorphism on $P\subset F_R$, and that it detects some intersections. In the next lemma, we extend it so $\phi$ agrees with a homeomorphism on $F_R\setminus N$.

\begin{lemma}\label{lemma:aaa}
  Let $S$ be a  closed surface of genus $g\geq 3$ and   $\phi: F_R \to \nonsep(S)$ a \lis map. There exists $h\in \mcg(S)$ such that $h|_{F_R\setminus N}=\phi|_{F_R\setminus N}$.
\end{lemma}
\begin{proof}
  By Corollary \ref{lemma:fr_top_typ} there exists $h\in \mcg(S)$ such that $h\circ \phi$ fixes every $p\in P$; we rename $h\circ \phi$ to $\phi$. Recall that by Remark $\ref{rmk:dist_two}$ we know $\phi$ is injective.

  First, we find a homeomorphism that agrees with $\phi$ on $c_1\in C$. Observe that $\phi(c_1)$ is contained in the torus $T_1$ bounded by $p_2^+$ and $p_2^-$ (see Figure \ref{fig:torus_tone}). Moreover,  Lemma \ref{lemma:detect_intersection_closed} imply there exist disjoint arcs 
  $a\in \phi(c_2)\cap T_1$ and $\tilde{a}\in \phi(c_{g+1})\cap T_1$. Notice both arcs have endpoints in $p_2^+$ and $p_2^-$, as otherwise $\phi(c_1)$ would not intersect $p_1$, contradicting Lemma \ref{lemma:detect_intersection_closed}. It follows that $\phi(c_1)$ is the curve contained in the annulus $T_1\setminus (a\cup \tilde{a})$. Even more, there exists a twist $h'\in \mcg(S)$ along $p_1$ and $p_2$, such that $h'\circ \phi(c_1)=c_1$. We rename $h'\circ \phi$ to $\phi$.

  \begin{figure}[h!]
    \centering
    \includegraphics[width=0.45\linewidth]{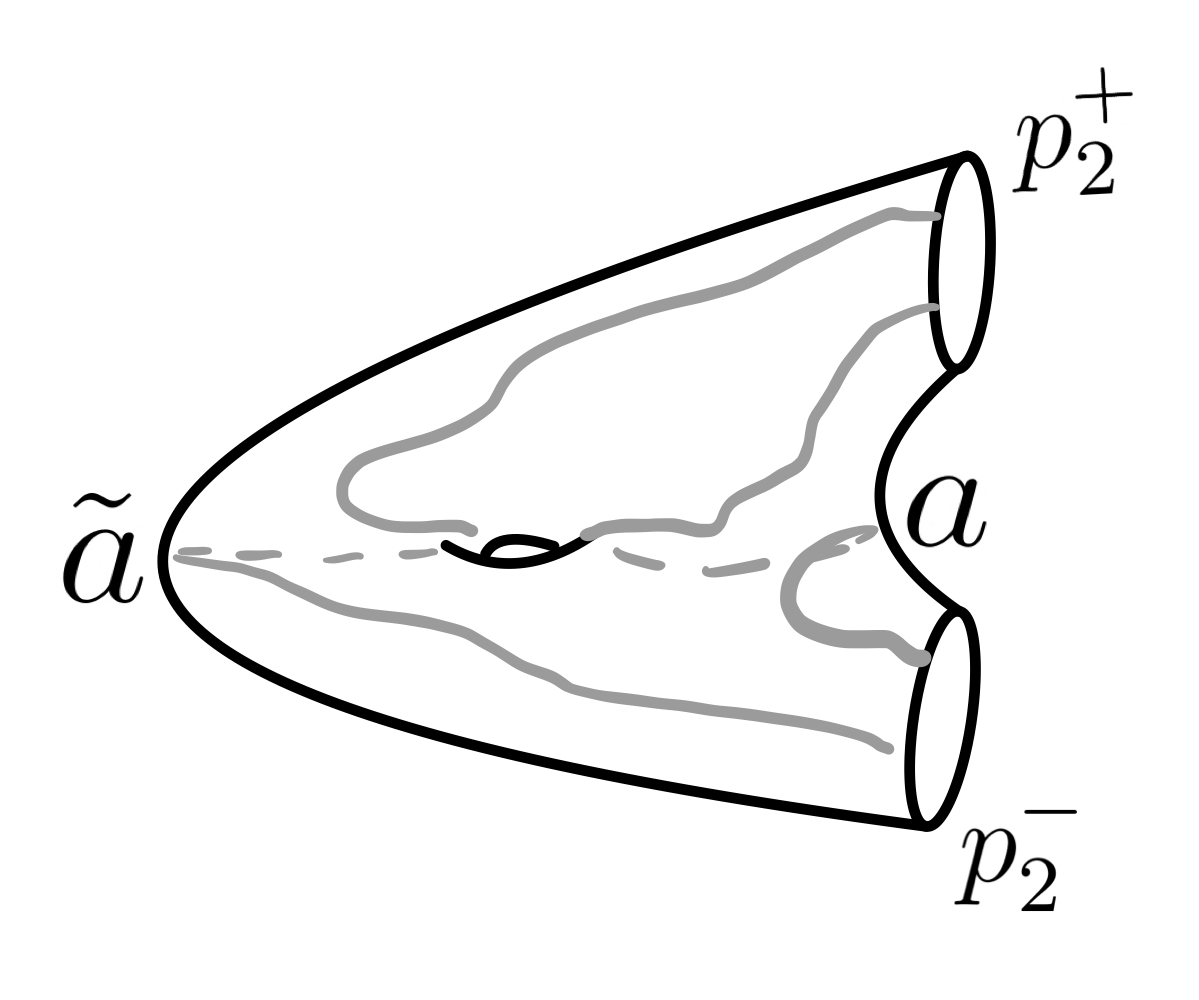}
    \caption{Torus $T_1$.}
    \label{fig:torus_tone}
  \end{figure}

  The same argument with minor changes yield homeomorphisms that agree with $\phi$ on every $c_k\in C\setminus\{c_{g+1}\}$. Thus, we may assume that $\phi$ fixes every curve in $ P\cup C\setminus\{c_{g+1}\}\subset F_R$. 
  
  To finish the proof, we are going to check that $\phi$ is fixing $c_{g+1}$ and every curve in $U \cup D \cup L\cup R$:
\begin{itemize}
  \item Notice that $\phi(c_{g+1})$ is contained in the torus with boundary $T' = S\setminus \bigcup_{i=2}^gp_i$. In fact, $\phi(c_{g+1})$ is the unique curve contained in the annulus $S'\setminus\bigcup_{k=1}^g c_k$, i.e, $\phi(c_{g+1})=c_{g+1}$. 
  \item To prove $\phi$ fixes $U$, consider $u_k\in U$. Observe that $\phi(u_k)$ is contained in the sphere $S_k^+$ bounded by the curves $p_{k-1}^+,\,p_k,\, p_{k+1}$ and $p_{k+1}^+$. Moreover, $\phi(u_k)$ must be the only curve in the annulus $S_k^+\setminus (c_k\cup c_{g+1})$, i.e, $\phi(u_k)=u_k$.
  \item To prove $\phi$ fixes $D$, consider $d_k\in U$. Notice that $\phi(d_k)$ is contained in the sphere $S_k^-$ bounded by the curves $p_{k-1}^-,\,p_k,\, p_{k+1}$ and $p_{k+1}^-$. Now, $\phi(d_k)$ must be the only curve in the annulus $S_k^-\setminus \{c_k, \,c_{g+1}\}$, i.e, $\phi(d_k)=d_k$.
  \item To prove $\phi$ fixes $L$, consider the curve $l_k\in L$. The image $\phi(l_k)$ is a curve in the subsurface $S_l$ bounded by $p_{k-1}^-,\,p_k,\, p_k^+,\, p_{k+1}^-$ and $p_{k+1}^+$. Since $\phi(l_k)$ is disjoint from curves in $C$, we know that $\phi(l_k)$ is contained in the pair of pants $Q = S_l\setminus (c_k\cup c_{k+1})$. Note that $Q$ has only one boundary component not already in $P$, and so we conclude $\phi(l_k)=l_k$. Naturally, a similar argument yields that $\phi$ fixes $r_k\in R$. 

\end{itemize}

  Summarizing, we have proven that there exists a mapping class $f\in \mcg(S)$ such that $f\circ \phi$ fixes $F_R\setminus N$. In other words, $f^{-1}|_{F_R\setminus N} = \phi|_{F_R\setminus N}$. 
\end{proof}

To finish the proof of Theorem \ref{thm:fr}, we need to check that we can take $\phi$ to agree with a homeomorphism on $N\subset F_R$, and that such homeomorphism is unique. Before doing so, we require one more definition:

The \emph{back-front (orientation reversing) involution} is the unique non-trivial mapping class  $\iota \in \mcg(S)$ fixing every curve $c\in P \cup C$.

\begin{proof}[Proof of Theorem \ref{thm:fr}  for closed surfaces]
  Let $\phi: F_R \to \nonsep(S)$ be a \lis map. By Lemma \ref{lemma:aaa} there exists $h\in \mcg(S)$ such that $h \circ \phi$ fixes $F_R\setminus N$, we rename $h\circ \phi$ as $\phi$. 
  
  Consider $nd\in N$ and notice $\phi(nd)$ is a curve in the genus two surface $T$ bounded by the curves $p_2^+,\, p_3,\, p_4,\, \dots, p_{g-1},\, p_{g-1}^+$. Since $\phi(nd)$ is disjoint from $C$, then $\phi(nd)$ is a curve in  $Q=T\setminus \bigcup_{c\in C} c$. Note that $Q$ is the disjoint union of two pairs of pants and it has only two boundary components not contained in $C$, namely, $nd$ and $\iota(nd)$, where $\iota$ is the back-front involution. It follows that $\phi(nd)\in \{nd,\, \iota(nd)\}$. Thus, by precomposing by $\iota$ if necessary, we may assume $\phi$ fixes $F_R\setminus \{nl, \, nr\}$.

  We continue by proving that $\phi$ fixes $nl$. Note that $\phi(nl)$ is a curve in the torus $T_1$ bounded by $p_2^-$ and $p_2^+$. Even more, $\phi(nl)$ is contained in $Q=T_1\setminus (c_1\cup nd)$, which is the union of an annulus and a pair of pants. Note that only one curve in $Q$ is not a curve in $P\cup C$, and so we deduce $\phi(nl)=nl$. An analogous argument leads to the conclusion $\phi(nr) =nr$. 

  So far, we  have found a composition of mapping classes $f\in \mcg(S)$ such that $f \circ \phi$ is the identity in $F_R$. Therefore, $\phi$ is induced by the mapping class $f^{-1}$.

  To prove the uniqueness of the inducing mapping class, suppose $f,\tilde{f}\in \mcg(S)$ both induce the same $\phi: F_R\to \nonsep(S)$. Since $g = \tilde{f}\circ f^{-1}$ fixes the curves $(P\cup C)\subset F_R,$ the Alexander method implies that $g$ is either the identity or the back-front orientation reversing involution. However, since $g$ also fixes the curve $nd$, then $g$ is the identity and  $f=\tilde{f}$. 
\end{proof}

\section{Finite rigid exhaustion for closed surfaces}\label{sec:exh_no_puncs}

In this section we prove that for any closed surface $S$ of genus $g\geq 3$, there exist subcomplexes $F_{1} \subset F_{2} \subset  \dots \subset \nonsep(S)$ such that each $F_{i}$ is a finite rigid set with trivial pointwise stabilizer and \[\bigcup_{i=1}^\infty F_{i} = \nonsep(S).\] 

The strategy to produce an  exhaustion is to first extend  $F_R$ to a larger finite rigid set $F_{1}$ with desirable properties. Then, the set $F_{1}$ will work as base case for an induction  that  enlarges $F_{i}$ into $F_{i+1}$. This method heavily resembles the proof given by Aramayona and Leininger in \cite{aramayona_exhausting_2016} to produce an exhaustion of the curve complex.

The plan of the proof is summarized in the following lemma (cf. \cite[Lemma 3.13]{aramayona_exhausting_2016}):

\begin{lemma}\label{lemma:create_exh}
  Let $S$ be a finite-type surface and let $F_{1} \subset \nonsep(S)$ be a a finite rigid set with trivial pointwise stabilizer and $\{h_1, \dots, h_k\}$ a set of generators of $\mcg(S)$. If $F_{1}\cup h_j(F_{1})$ is a  finite rigid set with trivial pointwise stabilizer for every $j$, then the sets   \[
     F_{i+1} = F_{i} \cup \bigcup_{j=1}^k \left( h_j(F_{i}) \cup  h_j^{-1}(F_{i}) \right),
  \]
  satisfy that $F_{1} \subset F_{2} \subset \dots \subset \nonsep(S)$, $\bigcup F_i  = \nonsep(S)$ and every $F_{i}$ is finite rigid with trivial pointwise stabilizer. 
\end{lemma}
\begin{proof}
  First, notice that if $F_{1}\cup h_j(F_{1})$ is a finite rigid set with trivial stabilizer, then the same holds for $F_{1}\cup h_j^{-1}(F_{1})$.

  We will now check that $F_{2}$ is finite rigid with trivial stabilizer: take $\phi: F_{2} \to \nonsep(S)$ and observe that by hypothesis there exists $f_j \in \mcg(S)$ such that \[\phi|_{F_{1}\cup h_j(F_{1})} = f_j|_{F_{1}\cup h_j(F_{1}),} \]for every $j\in \{1,\dots,k\}$. Since $F_{1}$ has trivial pointwise stabilizer and \[f_j |_{F_{1}} = \phi|_{F_{1}}=  f_k |_{F_{1}},\] we deduce  $f_j = f_k$ for all $j,k$. Clearly, the same argument holds when considering $F_{1} \cup h_j^{-1}(F_{1})$. It follows that $\phi = f$ for $f\in \mcg(S)$ and  $F_{2}$ is a finite rigid set with trivial pointwise stabilizer.

  For $F_{i+1}$ we proceed by induction. Consider $\phi: F_{i+1}\to \nonsep(S)$, by induction we have $\phi|_{F_{i}} = f|_{F_{i}}\in \mcg(S)$ and \[\phi|_{h_j(F_{i})} = f_j|_{h_j(F_{i})}\] for some $f_j\in \mcg(S)$. Observe that $F_{1} \subset F_{i} \cap h_j(F_{i})$ and $F_{1}$ is trivially stabilized, thus $f=f_j$ for every $j$ and so  $F_{i+1}$ is finite rigid with trivial stabilizer.
\end{proof}

To produce the finite rigid exhaustion of $\nonsep(S)$ we are going to use the previous lemma. In this direction, we enlarge $F_R$ into a set $F_{1}$ and provide a set of generators for $\mcg(S)$.

  \subsection{Enlarging the finite rigid set}\label{subsec:enlarging_finite_rigid_set_compact}
  
  Let $\iota \in \mcg(S)$ be the back-front orientation reversing involution. We define the set of curves $A := A_1 \cup A_2$, where:
  
  \[A_1 := \bigcup_{k=1}^{g-1} \partial(c_k, p_{k+1}, c_{k+1}).\] Note that the set $\partial(c_k, p_{k+1}, c_{k+1})$ has two curves, one curve in the front component and one curve in the back component of the surface. We will call $c_{k,k+1}$ the curve in the front component, so \[\partial(c_k, p_{k+1}, c_{k+1})=\{c_{k,k+1}, \,\iota(c_{k,k+1})\}.\]

We proceed to define $A_2$. Consider the torus $T$ bounded by  $p_{k-1}^-,\,p_{k+1}^-$ and $u_k$. Now, there is only one curve in the pair of pants $T\setminus (c_{k-1} \cup c_k \cup \iota(c_{k-1,\,k}))$ that is not already in $P\cup C$, denote this curve by $nl_k$ (see Figure \ref{fig:A2_nopuncs}). Analogously, let  $nr_k$ be the unique curve in $T\setminus (c_k\cup c_{k+1} \cup \iota (c_{k,\, k+1}))$ that is not already a curve in $P\cup C$ (see Figure \ref{fig:A2_nopuncs}). We set \[A_2 = \{nl_{k}, nr_{k}|\; k\in \{2, \dots, g-1\}\}.\]

  \begin{figure}[h!]
    \centering
    \begin{subfigure}{.45\linewidth}
      \centering
      \includegraphics[width=0.8\linewidth]{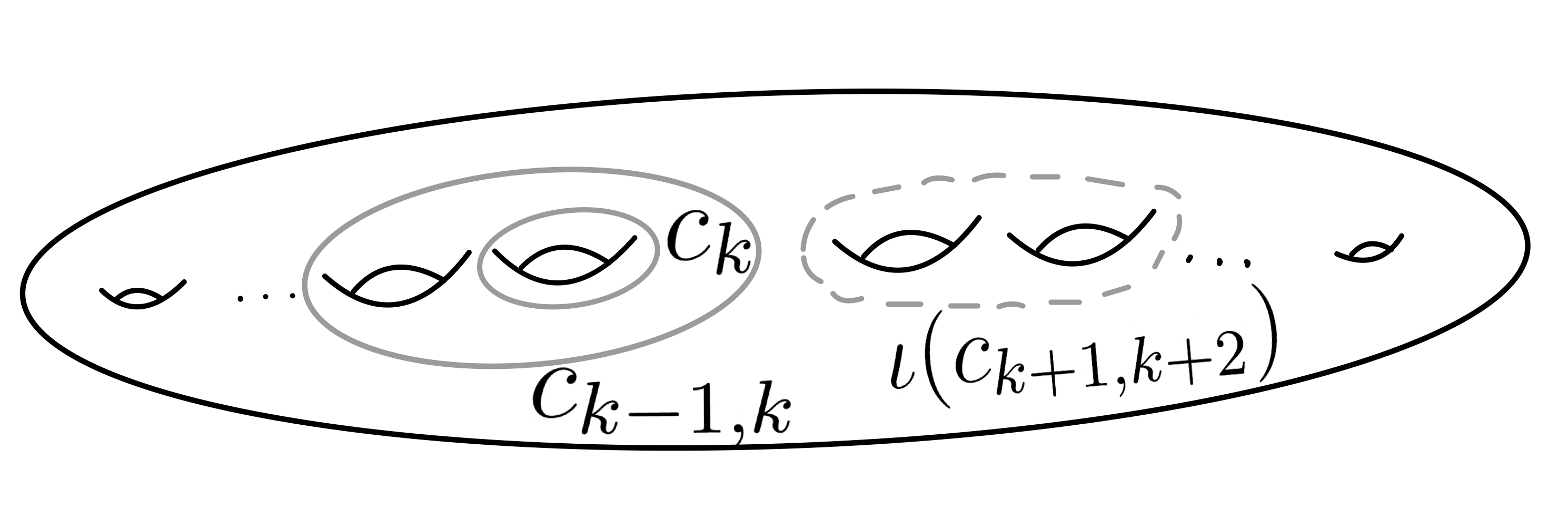}
      \caption{Curves in $A_1$.}
      \label{fig:A1_nopuncs}
    \end{subfigure}
    \begin{subfigure}{.45\linewidth}
      \centering
      \includegraphics[width=0.8\linewidth]{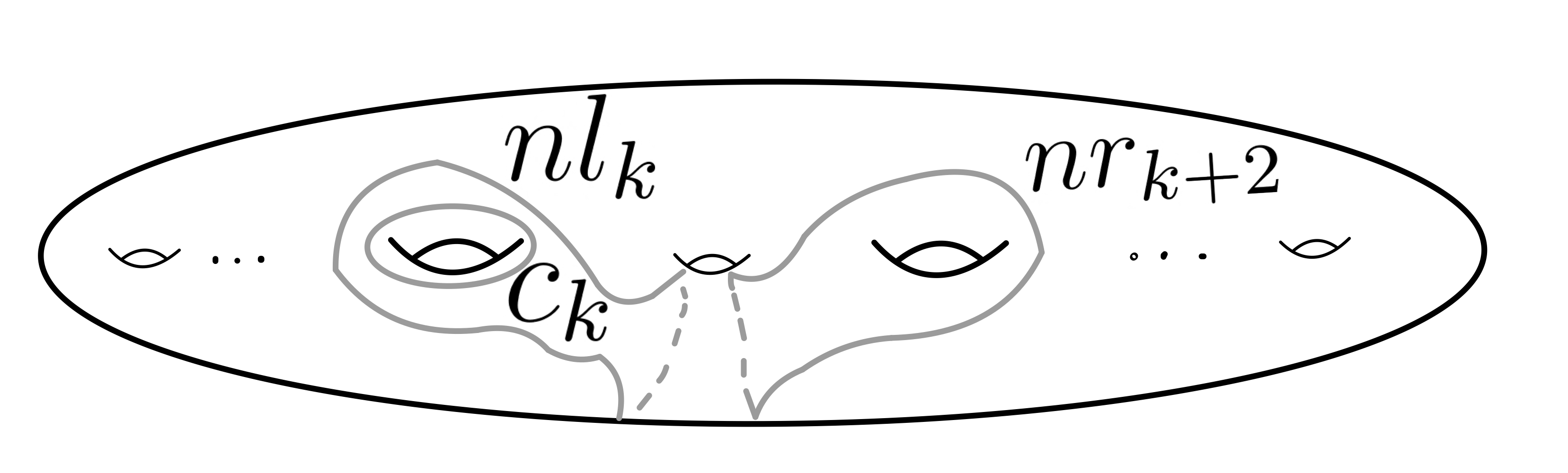}
      \caption{Curves in $A_2$.}
      \label{fig:A2_nopuncs}
    \end{subfigure}
    \caption{Curves in $A$ for a closed surface.}
  \end{figure}

  \begin{lemma}\label{lemma:enlarge_fr_compact}
    Let $S$ be a closed surface of genus $g\geq 3$. The subcomplex spanned by $F_R \cup A \subset \nonsep(S)$ is a finite rigid set with trivial pointwise stabilizer.
  \end{lemma}
  \begin{remark}
    In the  argument below we sometimes abuse notation using $ p_1^+= p_1^- = p_1$ and $ p_g^+=  p_g^- = p_{g+1}$.
  \end{remark}
  \begin{proof}
    Take $\phi: F_R \cup A \to \nonsep(S)$ a \lis map. By  Theorem \ref{thm:fr} in the closed case, there is an $h\in \mcg(S)$ such that $h\circ \phi$ fixes $F_R$; we rename $h\circ \phi$ to $\phi$. Furthermore, the subcomplex $F_R \cup A$ has diameter two, so $\phi$ is injective.

    First, we  check $\phi$ fixes  $A_1$. Note that $ \phi(\iota (c_{1,2}))$ is contained in the torus $T$ bounded by the curves $p_2,\, p_3^+$ and $p_3^-$. Since $\phi(\iota(c_{1,2}))$ is disjoint from the curves in $C$, we have that  $ \phi(\iota (c_{1,2}))$  is contained in the pair of pants $Q=T\setminus (c_1\cup c_2\cup c_3)$. Notice that only one curve in $Q$ is non-separating and disjoint from $nd$, namely $\iota(c_{1,2})$ and so  $\phi(\iota(c_{1,2}))=\iota(c_{1,2})$. Now, $\phi(c_{1,2})$ is the unique non-separating curve in $Q$ distinct from $\iota(c_{1,2})$, that is, $\phi(c_{1,2})=c_{1,2}$.
     
    By an obvious modification of the above argument one checks  that \[\phi(c_{k,k+1}) = c_{k,k+1}\text{ and }\phi(\iota  (c_{k,k+1}))=\iota(c_{k,k+1}),\]  for $k\leq g-1$.

    We are left to check that $A_2$ is also fixed by $\phi$. Take   $nr_k\in A_2$; we have that $\phi(nr_k)$ is contained in the torus $T$ bounded by the curves $p_{k-1}^-,\, p_{k+1}^-$ and $u_k$. Also, $\phi(nr_k)$ is disjoint from $c_k$, $c_{k+1}$ and $\iota(c_{k,k+1})$, thus $\phi(nr_k)$ is a curve in the pair of pants $Q=T\setminus (c_k\cup c_{k+1} \cup \iota(c_{k,k+1}))$. Notice that the only curve in $Q$ distinct from curves in $P \cup C$ is the curve $nr_k$, thus we have $\phi(nr_k) = nr_k$. For $nl_k \in A_2$ the argument is analogous.

    We have proven that if $\phi$ fixes $F_R$, then it fixes $F_R \cup A$. The statement follows immediately.
  \end{proof}
  
  We define $F_{1}:= F_R \cup A$ and proceed to prove that it satisfies the hypotheses of Lemma \ref{lemma:create_exh}. 

  \subsection{Constructing the exhaustion for closed surfaces}

  Let $\iota$ be the back-front orientation reversing involution and let $\delta_\alpha$ be the right Dehn twist along the curve $\alpha$. The set of Dehn twists \[H := \{\delta_{p_1}, \delta_{p_2}, \dots, \delta_{p_{g+1}}, \delta_{{p_2^-}}\} \cup \{\delta_{c_1}, \dots, \delta_{c_k}\}\] are known as the \emph{Humphries generators}, which are known to generate the group of orientation preserving mapping classes $\mcgo(S)$  \cite{humphries_generators_1979}. Hence, $H \cup \{\iota\}$ generates $\mcg(S)$. 

  \begin{lemma}\label{lemma:fr1_as_in_lemm}
    Let $h\in H\cup \{\iota\}$. The set $F_{1} \cup h(F_{1})$ is a finite rigid set with trivial pointwise stabilizer. 
  \end{lemma}
  \begin{proof}
    First, we prove it for $h=\delta_{p_1}$. Let $\phi: F_{1} \cup h(F_{1}) \to \nonsep(S)$ be  any \lis map.
    
    Since $F_{1}$ is a finite rigid set, we may assume that $\phi$ fixes $F_{1}$ by precomposing with a mapping class. Moreover, since $h(F_{1})$ is also finite rigid, there exists a mapping class $f\in \mcg(S)$ such that \[\phi|_{h(F_{1})}= f|_{h(F_{1})}.\]
    
    Proving $F_{1} \cup h(F_{1})$ is a finite rigid set with trivial pointwise stabilizer  boils down to prove that $f = 1 \in \mcg(S)$.

    Recall that $p_1$ is the  associated curve to the Dehn twist $h=\delta_{p_1}$. Now, let $S' = S\setminus p_1$ and consider the well-known cutting homomorphism (see \cite[Proposition 3.20]{farb_primer_2012}) \[1 \rightarrow \langle \delta_{p_1} \rangle \rightarrow \mcgo(S, p_1) \rightarrow \mcgo(S'), \]
    where $\mcgo(S') \subset \mcg(S')$ is the subgroup of orientation preserving classes and $\mcgo(S,p_1) \subset \mcg(S)$ is the subgroup of orientation preserving  classes that fix $p_1$. Notice that $f$ fixes all the curves in $F_{1} \cap h(F_{1})$,  which  fill $S'$. In particular, $f$ is orientation preserving and fixes $p_1$. Furthermore, by means of the Alexander method, the image of $f$  by the cutting homomorphism is trivial and the above sequence implies that $f=\delta_{p_1}^k$, for some $k\in \mathbb{Z}$. We are left to see that $k=0$.

    Note that $i\left(nl, \delta_{p_1}(c_{g+1})\right)=0$. As $\phi$ is locally injective, we know that $i\left(\phi(nl), \phi\left(\delta_{p_1}(c_{g+1})\right)\right)=0$, which is equivalent to 
    \begin{equation}\label{eq:intersect_0}
      i\left(nl, \delta_{p_1}^{k+1}(c_{g+1})\right)=0,  
    \end{equation}
    since $\phi|_{F_{1}}=\text{id}$ and $\phi|_{h(F_{1})} = \delta_{p_1}^k$.
    
    It can be directly checked that   Equation (\ref{eq:intersect_0}) is satisfied if and only if $k=0$, thus proving the statement of the lemma for $h=\delta_{p_1}$.

    The rest of the cases $h\in H$ are proved in exactly the same way,  but changing  the curves in Equation (\ref{eq:intersect_0}). We list all the cases for completeness:

    \begin{itemize}
      \item If $h=\delta_{p_{g+1}}$, change   $(nl, c_{g+1})$ in (\ref{eq:intersect_0}) by $(c_{g+1}, nr)$;
      
      \item If $h=\delta_{p_i}$ for $i\in\{2,\dots, g-2\}$, change $(nl,c_{g+1})$ by $(nl_i, c_{i,i+1})$;
      
      \item If $h=\delta_{p_{g-1}}$, change $(nl,c_{g+1})$ by  $(c_{g-2,g-1}, nr_{g-1})$;
      
      \item If $h=\delta_{{p_2^-}}$, change $(nl,c_{g+1})$ by $(c_{g+1}, nr_2)$;
      
      \item If $h=\delta_{c_k}$ for $k\in \{1,\dots, g-1\}$, change $(nl,c_{g+1})$ by $(p_{k+1}, nr_{k+1})$;
      
      \item If $h=\delta_{c_g}$, change $(nl,c_{g+1})$ by $(nl_{g-1}, p_g)$.

    \end{itemize}

    To finish the proof, we need to consider the case $h=\iota$. In this case, $F_{1} \cap h(F_{1})$ contains the trivially pointwise stabilized set $P\cup C  \cup A_1$, so it follows immediately that $f=1\in \mcg(S)$. 
  \end{proof}
  
\begin{proof}[Proof of Theorem \ref{thm:exh} for closed surfaces]
  Let $S$ be a closed surface of genus $g\geq 3$. Consider $F_{1}:=F_R\cup A$ and the set of generators $H\cup\{\iota\}$ of the extended mapping class group $\mcg(S)$. By Lemma \ref{lemma:create_exh} we have the desired exhaustion, where the hypotheses have been checked in lemmas \ref{lemma:enlarge_fr_compact} and \ref{lemma:fr1_as_in_lemm}.
\end{proof}

\section{Finite rigid sets for punctured surfaces}\label{sec:fr_puncs}

Let $S=S_{g,n}$ with $n\geq 1$ and  $g\geq 3$. In this section we construct a finite rigid set $F'_R\subset \nonsep(S)$ with trivial pointwise stabilizer. As we will see, the argument to prove the rigidity of $F'_R$ is morally the same  as the one in Section \ref{sec:fr_no_puncs}, although  more involved as it requires to manipulate more curves.

\subsection{Constructing the finite rigid set}\label{sec:constructing_fr_puncs}
We start by naming some curves in the punctured surface $S$.

Consider the closed surface $S_g$ and the pair of pants $Q$ bounded by $p_1,\,p_2, \,p_2^+$. Note that $Q\setminus d_2$ has two connected components, name $A$ the connected component intersected by $c_{g+1}$. Now, $A\setminus (c_1 \cup c_{g+1})$ is the union of two disks one in the front component and one in the back component of $S_g$, name $B$ the disk in the back component $(S_g\setminus \bigcup c_i)^-$. By removing $n$ points from the interior of $B$ we obtain the punctured surface $S$. Via this procedure we get natural analogues  in $S$  of the sets of curves $P$, $C$, $U$, $D$, $L$, $R$ and $N$. 

After this, we define top of the surface $(S\setminus \bigcup p_i)^+$, the bottom $(S\setminus \bigcup p_i)^-$, the front $(S\setminus \bigcup c_i)^+$  and the back $(S\setminus \bigcup c_i)^-$, in the same way we did in Section \ref{sec:fr_no_puncs}. 

The set $P$ is no longer a pants decomposition in $S$; to fix this we add some curves. First, rename $p_2^+$ to $p_{2,n}^+$ and consider the pair of pants $Q$ bounded by $p_1,\,p_2$ and $p_{2,n}^+$. Observe $Q\setminus c_1$ is an annulus with $n$ punctures. Now, fix a set of non-separating  curves $\{p_{2,0}^+,\,p_{2,1}^+, \,\dots,\, p_{2,n-1}^+,\, p_{2,n}^+\}$ in $Q\setminus c_1$  satisfying that $p_{2,i}^+$ and $p_{2,i+1}^+$ bound a once punctured annulus. Adding the curves $\{p_{2,0}^+,\,p_{2,1}^+, \,\dots, p_{2,n-1}^+,\, p_{2,n}^+\}$  to the set $P$, makes it a pants decomposition (see Figure \ref{fig:pant_decomposition_with_puncs}). We will say that the puncture contained in the annulus bounded by $p_{2,k-1}^+$ and $p_{2,k}^+$ is the $k$th puncture.

In this setting, we also define the back-front (orientation reversing) involution as the unique non-trivial element $\iota \in \mcg(S)$ that fixes every curve $P\cup C$.

Let $pl_1$ be the unique curve in the punctured pair of pants bounded by $p_1,\,p_2,\, p_{2,1}^+$ that is disjoint from $c_2$ and is not a curve already in $P$. Similarly, let $pl_k$ be the unique curve in the punctured pair of pants bounded by $pl_{k-1},\,p_2,\,p_{2,k}^+$ that is disjoint from $c_2$ and is not already a curve in $P$ (see Figure \ref{fig:left_curves_with_puncs}). We set \[Pl:=\{pl_1, \dots, pl_n\}.\]

Let $pr_n$ be the unique curve in the punctured pair of pants bounded by $p_{2,n-1}^+,\,p_3,\, p_3^+$ that is disjoint from $c_2$ and is not already a curve in $P$. Inductively, define $pr_{k}$ to be the unique curve in the punctured pair of pants bounded by $p_{2,k-1}^+,\,p_3,\,p_3^+$ that is disjoint from $c_2$ and not already in $P$ (see Figure \ref{fig:right_curves_with_puncs}). We set 
\[Pr := \{pr_1, \dots, pr_n\}.\]

Lastly, we define the set
\[C_* = \partial(c_1,\,p_2,\, c_2) \cup \partial (c_2,\,p_2,\dots,\,p_g,\,c_g).\] Note  the set $\partial(c_1,\,p_2,\, c_2)$ has two curves, one curve in the front component which we will denote by $c_{1,\,g}$ and one curve in the back component that is $\iota(c_{1,g})$. Also, the set $\partial (c_2,\,p_2,\dots,\,p_g,\,c_g)$  has two curves, denote by $c_{2,\,g}$ the curve in the front component and by $\iota(c_{2,\,g})$ the curve in the back component (See Figure \ref{fig:cstar}).

\begin{figure}[h!]
  \centering
  \begin{subfigure}{.45\linewidth}
    \centering
    \includegraphics[width=0.8\linewidth]{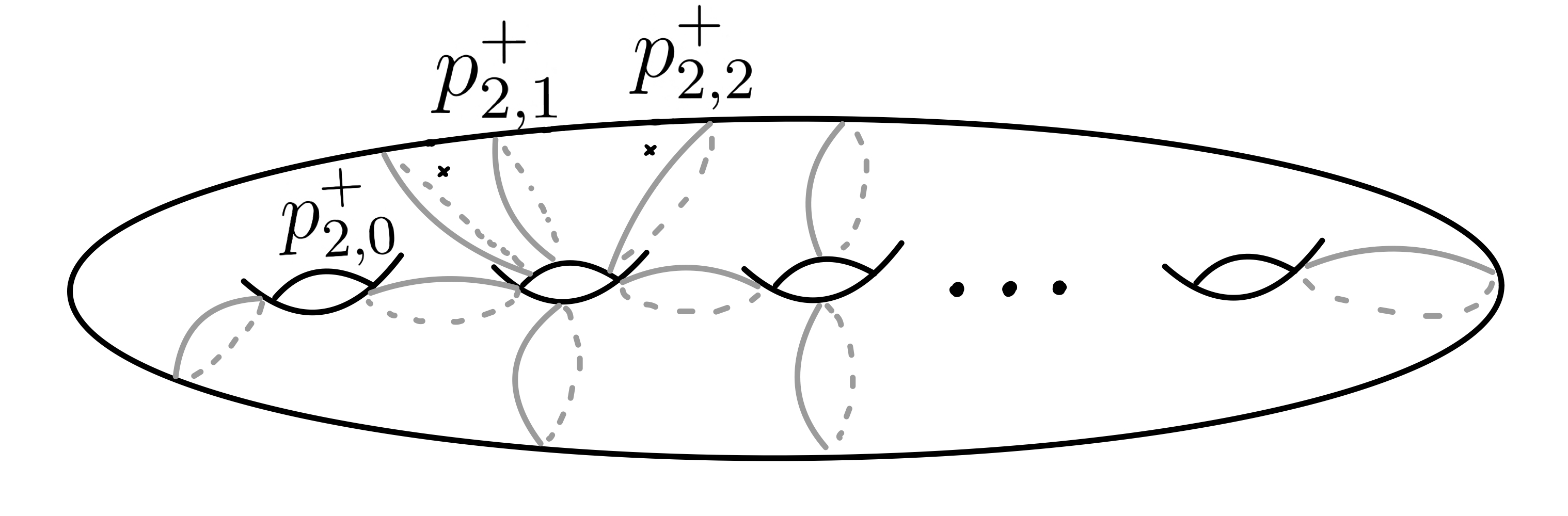}
    \caption{Pants decomposition $P$ for $S_{g,2}$.}
    \label{fig:pant_decomposition_with_puncs}
  \end{subfigure}
  \begin{subfigure}{.45\linewidth}
    \centering
    \includegraphics[width=0.8\linewidth]{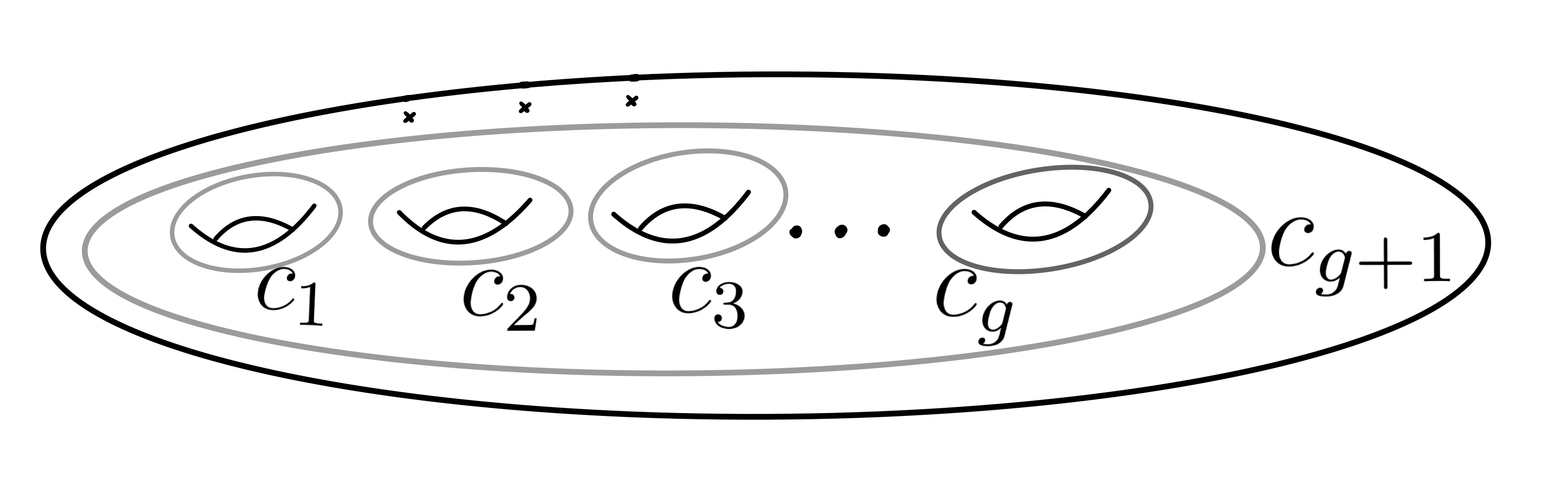}
    \caption{Curves in  $C$ for $S_{g,3}$.}
    \label{fig:c_with_puncs}
  \end{subfigure}\par \medskip

  \begin{subfigure}{.45\linewidth}
    \centering
    \includegraphics[width=0.8\linewidth]{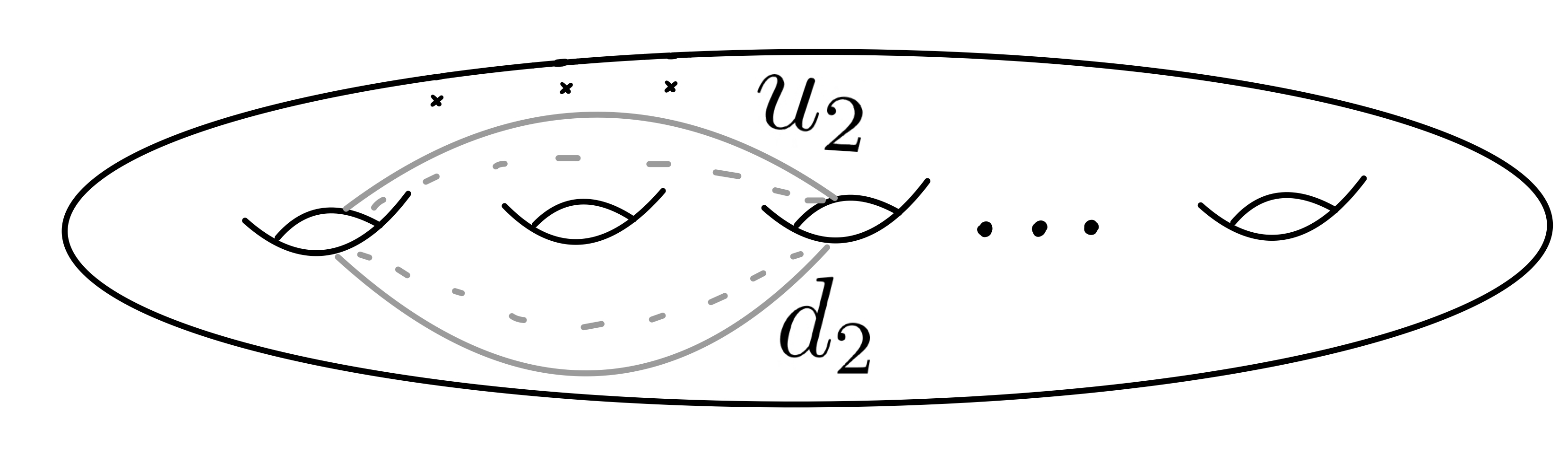}
    \caption{Up and down curves  $u_2, d_2$  for $S_{g,3}$.}
    \label{fig:ctilde_with_puncs}
  \end{subfigure}
  \begin{subfigure}{.45\linewidth}
    \centering
    \includegraphics[width=0.8\linewidth]{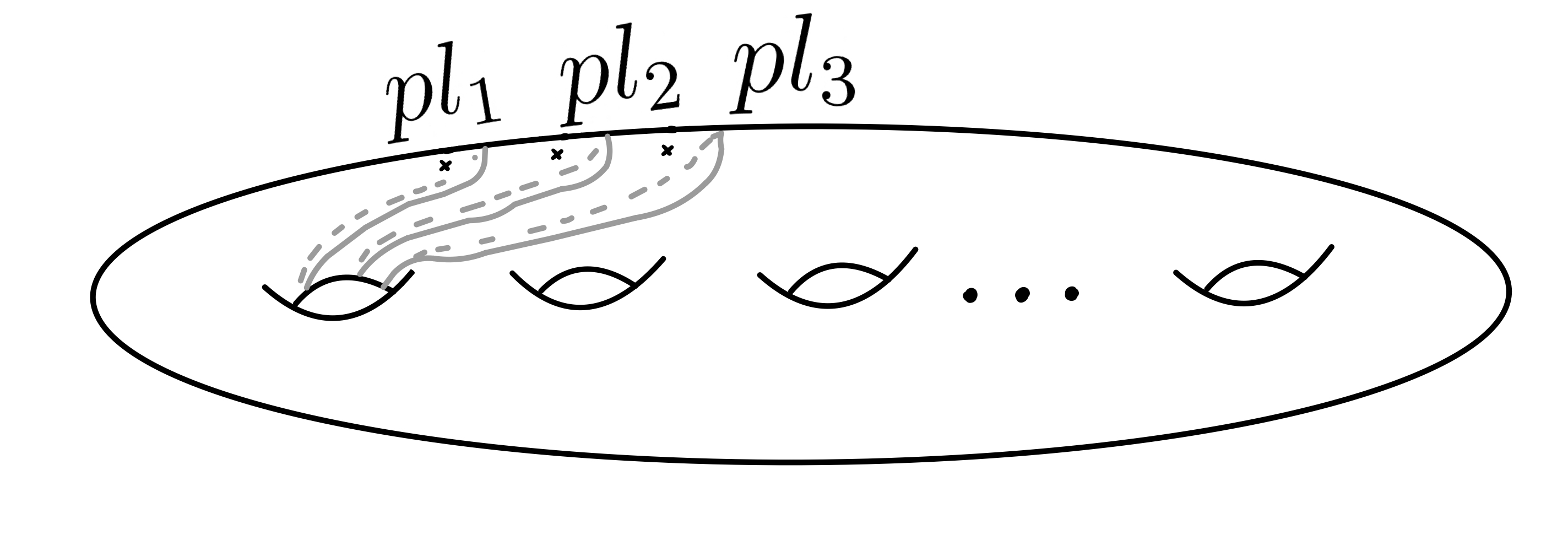}
    \caption{Curves $pl_1,pl_2, pl_3$ for $S_{g,3}$.}
    \label{fig:left_curves_with_puncs}
  \end{subfigure}\par\medskip

  \begin{subfigure}{.45\linewidth}
    \centering
    \includegraphics[width=0.8\linewidth]{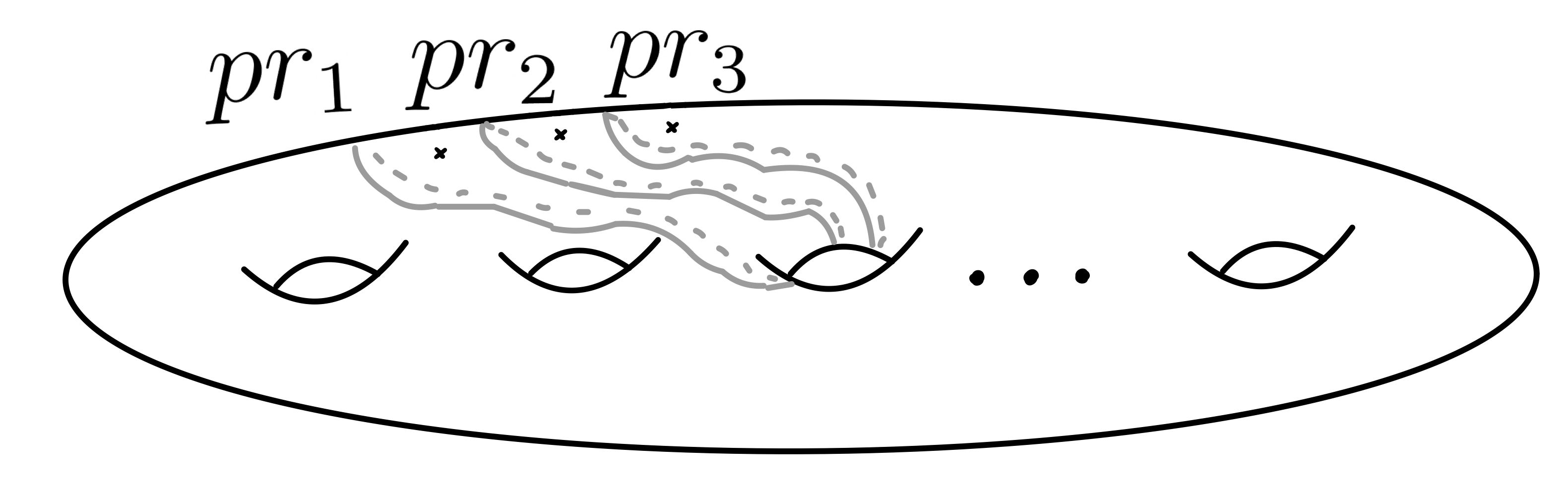}
    \caption{Curves $pr_1,pr_2, pr_3$ for $S_{g,3}$.}
    \label{fig:right_curves_with_puncs}
  \end{subfigure}
  \begin{subfigure}{.45\linewidth}
    \centering
    \includegraphics[width=0.8\linewidth]{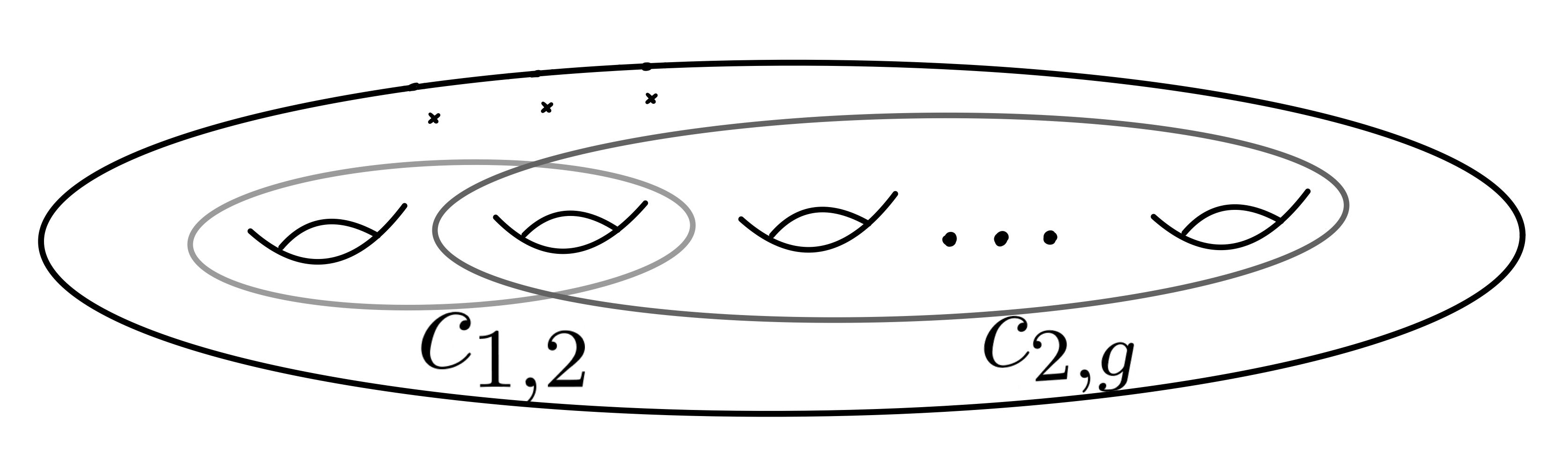}
    \caption{Curves $c_{1,2}$ and $c_{2,g}$ for $S_{g,3}$.}
    \label{fig:cstar}
  \end{subfigure}

  \caption{Curves in $F'_R$. Punctures are marked as small crosses on the top of the surface.}
\end{figure}

The finite rigid set $F'_R$ is the maximal subcomplex of $\nonsep(S)$ spanned by the vertices \[P \cup C\cup U \cup D \cup L\cup R \cup N \cup Pl \cup Pr  \cup C_*.\] 

\begin{remark}
  Note that the subcomplex  $F'_R$  has diameter two. Thus, any \lis map $\phi: F'_R \to \nonsep(S)$ is injective. 
\end{remark}

The following lemmas prove the finite rigidity of $F'_R$. 

\begin{lemma}\label{lemma:fr2_as_in_lemm}
  Let $S$ be a punctured surface of genus $g\geq 3$. For any \lis map $\phi: F'_R \to \nonsep(S)$, there exists $h\in\mcg(S)$ such that $h|_P=\phi|_P$.
\end{lemma}
\begin{proof}
  First, notice that we can extend Lemma \ref{lemma:compact_non_adj_rel_p} to the case of punctured surfaces, and the proof works with minor changes. 

  Lemma \ref{lemma:compact_adj_rel_p} can also be extended to punctured surfaces. The proof for once punctured surfaces is similar to the closed case. Here we prove it for surfaces with  $n>1$ punctures:

  Using Remark \ref{rmk:at_least_adj_curves}, it is straightforward to deduce that the curve $\phi(p_{2,i}^+)$ is adjacent rel to $P$  to $\phi(p_{2,i-1}^-)$ and $\phi(p_{2,i+1})$, for $i\in \{1,\dots, n-1\}$. As a consequence of these adjacencies, we deduce that $\phi(p_{2,0}^+)$ and $\phi(p_{2,n}^+)$ bound a subsurface with two boundary components and $n$ punctures. It follows that $\phi(p_{2,0}^+)$ is adjacent rel to $P$ to at least three curves. On the other hand, the non adjacency rel to $P$ of $\phi(p_{2,0}^+)$ implies that it is adjacent to at most three curves. Thus, $\phi(p_{2,0}^+)$ is adjacent to exactly three curves, namely $\phi(p_{2,1}^+),\, \phi(p_1)$ and $\phi(p_2)$. The rest of the adjacencies follow  just as in the proof of Lemma \ref{lemma:compact_adj_rel_p}.

  Once we know $\phi$ preserves adjacency rel to $P$, we can use Corollary \ref{lemma:fr_top_typ} to produce  a mapping class $h\in \mcg(S)$ with $h|_P=\phi|_P$.
\end{proof}

The next lemma proves that we may choose $h\in \mcg(S)$ to coincide with $\phi$ on all the curves of $F'_R$.

\begin{lemma}\label{lemma:fr2_curve_characterize}
  Let $S$ be a punctured surface of genus $g\geq 3$. For any \lis map $\phi: F'_R \to \nonsep(S)$ there exists $h\in\mcg(S)$ such that $\phi = h|_{F_R'}$.
\end{lemma}
\begin{proof}
   The idea is  to progressively detect intersections between curves and, by composing with Dehn twists along $P$, construct a mapping class that coincides with $\phi$ on $F'_R$. 
  
  First, by Lemma \ref{lemma:fr2_as_in_lemm}, there exists $h\in \mcg(S)$ such that $h\circ \phi$ fixes $P$ pointwise; rename $h \circ \phi$ as $\phi$.

  Next, using the same arguments as in lemmas \ref{lemma:intersections_1}, \ref{lemma:aux10} and \ref{lemma:detect_intersection_closed}, we detect the following intersections:
  \begin{itemize}
    \item intersection of curves in $U\cup D$ with curves in $P$,
    \item intersections of curves in $Pl$ with curves in $P$,
    \item intersections of curves in $Pr$ with curves in $P$, and
    \item $c_k$ with curves in $P$ for $k\in \{2, \dots,g-1\}$.
  \end{itemize}
  Notice that we can now use the proof of Lemma \ref{lemma:aaa} to produce a mapping class that coincides with $\phi$ on the curves of $U$, $D$ and  $\{c_2,\dots, c_{g-1}\}$. Thus, we may assume these curves are fixed by $\phi$. 
  
It follows that $\phi$ fixes the curves in $Pl$. To see this, note that $\phi(pl_1)$ is a curve contained in the punctured sphere $S'$ bounded by $p_1,\,u_2,\,p_3^+$, and  is disjoint from $p_{2,i}^+$ for $i>0$. Thus, $\phi(pl_1)$ is contained in the once punctured annulus which is a component of $S'\setminus p_{2,1}^+$. As there is only one curve in that annulus that is not already in $P$, it follows that $\phi(pl_1)=pl_1$. In the same way, $\phi(pl_k)$ is contained in the punctured sphere $S'_k$ bounded by $pl_{k-1},\,u_2,\,p_3^+$, and  is disjoint from $p_{2,i}^+$ for $i>k-1$. Thus, $\phi(pl_k)$ is contained in the once punctured annulus which is a component of $S'_k\setminus p_{2,k}^+$, and since there is only one curve in that annulus which is not in $P$ we conclude that $\phi(pl_k)=pl_k$. Naturally, an analogous argument works to prove that $\phi$ fixes every curve in $Pr$.

  Now, note that $\phi$ detects the  following intersections:
  \begin{itemize}
    \item curves in $C_*$ with  curves in $P$,
    \item curves in $\{c_1,c_g\}$ with curves in $P$.
  \end{itemize}
  For instance, consider $c_{1,2}\in C_*$ and $p\in P$ disjoint from $c_{1,2}$. If $\phi(c_{1,2})$ was disjoint from $p$, then $\phi(c_{1,2})$ would have to intersect either the curve  $\phi(c_1),\phi(c_2)$ or $\phi(\iota(c_{2,g}))$, leading to a contradiction. Thus, $\phi(c_{1,2})$ and $p$ intersect. 

  Using the above intersections and fixed curves, we now focus on finding a homeomorphism that agrees with $\phi$ on the curves in $C_*$. 
  
  Observe $\phi(c_1)$ (resp. $\phi(c_g)$)  is  contained in the subsurface $S'=S_{1,2}$ bound by the curves $p_2^-, p_2^+$  (resp. by $p_{g-1}^-,p_{g-1}^+$). Additionally,   since $\phi(c_1)$ is disjoint from the arcs $a_1=S'\cap \phi(c_{1,2})$ and $a_2= S' \cap \phi(c_2)$ (resp. $S'\cap \phi(c_{2,g})$ and $S' \cap \phi(c_{g-1})$), we have that $\phi(c_1)$ (resp. $\phi(c_g)$) is contained in the annulus $S'\setminus (a_1\cup a_2)$. It follows that $\phi(c_1)$ (resp. $\phi(c_g)$) is the unique curve in the annulus. We may  again consider a mapping class $h\in \mcg(S)$ that is a composition of  twists along curves in $P$ such that  $h\circ \phi(c)=c$ for $c\in \{c_1, c_g\}$; we rename $h\circ \phi$ to $\phi$ (the same argument with more details is given in Lemma \ref{lemma:aaa}).

  For $c_{1,2}\in C_*$, note that $\phi(c_{1,2})$  is contained in a subsurface $S' = S_{2,2}$ bound by $pr_1$ and $p^-_3$ (or $p_4$ if $g=3$). Even more, $\phi(c_{1,2})$ is contained in the pair of pants $S'\setminus (c_1\cup c_2 \cup c_3  \cup p_2)$, therefore it is one of the boundary components. One of the boundaries is a separating curve, so either $\phi(c_{1,2})=c_{1,2}$  or $\phi(c_{1,2})=\iota(c_{1,2})$.  These two alternatives are related by the involution $\iota$.  Thus, by precomposing with $\iota$ we can assume $\phi$ fixes $c_{1,2}.$

  It is now easy (but lengthy) to check that the curves in \[\{ \iota(c_{1,2}), \iota( c_{2,g}), nl,nr, nd \} \cup L \cup R\] are fixed by $\phi$. Thus, we have found a mapping class $h\in \mcg(S)$ such that $h\circ \phi$ fixes $F'_R$ and the statement follows.
\end{proof}

We have essentially completed the proof of Theorem \ref{thm:fr}.

\begin{proof}[Proof of Theorem \ref{thm:fr} for surfaces with punctures]
  Let $\phi: F_R' \to \nonsep(S)$ be a \lis map, Lemma \ref{lemma:fr2_curve_characterize} provides the mapping class inducing $\phi$. 
  
  The uniqueness of the mapping class follows as in the closed case.
\end{proof}

\section{Finite rigid exhaustion for punctured surface}\label{sec:exh_puncs}

Let $S$ be a punctured surface of genus $g\geq 3$. In this section we construct a sequence $F'_{1} \subset F'_{2} \subset  \dots \subset \nonsep(S)$ such that each $F'_{i}$ is a finite rigid set with trivial pointwise stabilizer and \[\bigcup_{i=1}^\infty F'_{i} = \nonsep(S).\] 

The strategy to construct the exhaustion is the same as in the closed case (see Section \ref{sec:exh_no_puncs}). First, we are going to enlarge the finite rigid set $F'_R$  to $F'_{1}$ and then use  Lemma \ref{lemma:create_exh} to construct the exhaustion.

\subsection{Enlarging the finite rigid set}

First, we enlarge $F'_R$:

Consider the set of curves $A_1$ and $A_2$ in the closed surface. Via the same procedure described in Section \ref{sec:constructing_fr_puncs}, we remove $n$ points from the interior of the closed surface $S_g$ and so we obtain the punctured surface $S$. This produces natural analogues of the set of curves $A_1$ and $A_2$ in the surface $S$.

We define the set of curves $A_3:=\{a_1,\dots,a_n\}$, where $a_k$ is the unique curve in the torus bounded by $pl_{k-1},\,pr_{k+1}$ and $d_2$, which is disjoint from $c_1,\,c_2,\,c_3$ and $\iota(c_{1,2})$ (see Figure \ref{fig:a3_with puncs}). 

We set $A':=A_1\cup A_2 \cup A_3$.

\begin{figure}[h]
  \centering
  \begin{subfigure}{.45\linewidth}
    \centering
    \includegraphics[width=0.8\linewidth]{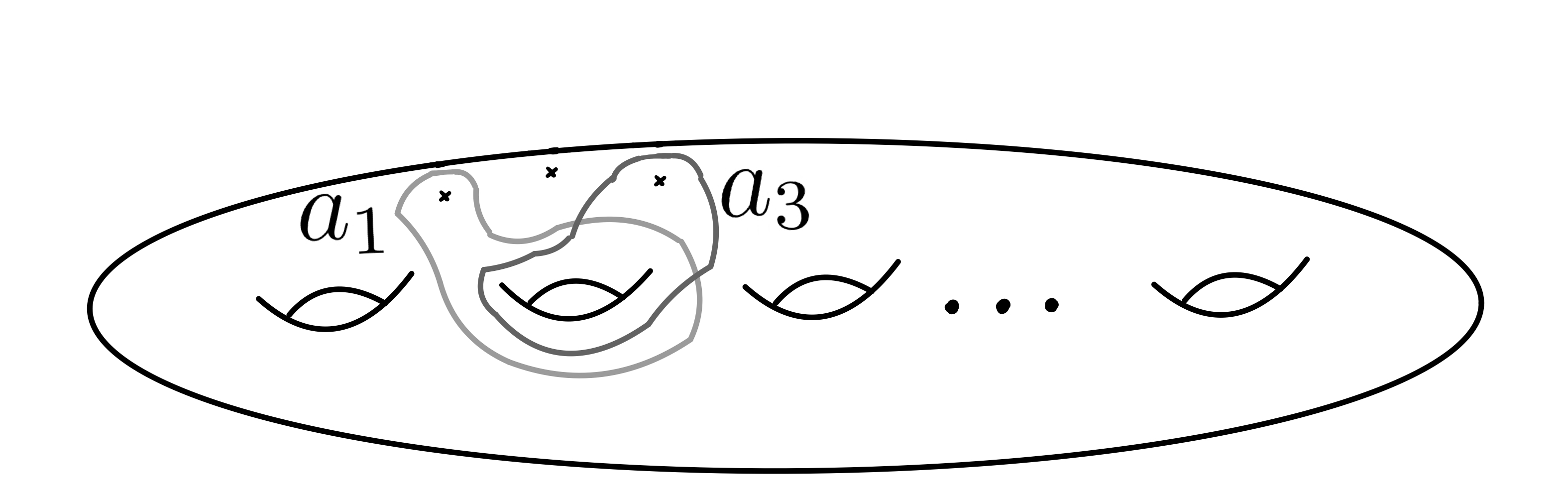}
    \caption{Curves $a_1,a_3$ in $A_3$ for $S_{g,3}$.}
    \label{fig:a3_with puncs}
  \end{subfigure}\par\medskip
\end{figure}

\begin{lemma}\label{lemma:almostlast}
  The set $F'_R \cup A'$ is finite rigid with trivial pointwise stabilizer.
\end{lemma}
\begin{proof}
  Let $\phi:F'_R \cup A' \to \nonsep(S)$ be a \lis map. By precomposing with a mapping class we may assume that $\phi$ fixes $F'_R$ pointwise. 

  First, we prove that $A_1$ is also fixed by $\phi$. The curves $c_{1,2}, \iota( c_{1,2}) \in A_1$ are already in $C_* \subset F_R'$, so they are fixed by $\phi$. For $c_{2,3} \in A_1$, notice that $\phi(c_{2,3})$ is contained in the sphere $S'$  bounded by $p_1, \,pl_n,\, c_2,\,c_3,\,p_4^+$ and $p_4^-$. Moreover, $\phi(c_{2,3})$ is contained in the pair of pants $S' \setminus c_1 \cup p_3 \cup \iota(c_{2,g})$. But there is only one curve in that pair of pants that is non-separating, i.e,  $\phi(c_{2,3})=c_{2,3}$. Slight modifications yield that $\phi(\iota(c_{2,3}))=\iota(c_{2,3})$. For the rest of $A_1$,  one can proceed as in the closed case (see the proof of Lemma \ref{lemma:enlarge_fr_compact}).

  To prove $A_2$ is fixed, we can just repeat the argument as in the closed case (see Lemma \ref{lemma:enlarge_fr_compact}).

  To finish the proof we must show that $a_k \in A_3$ is  fixed. To check this, note that $\phi(a_k)$  is contained in the torus $T$ bounded by  $pl_{k-1}, \,pr_{k+1}$ and $d_2$. Note that $\phi(a_k)$ is the unique curve in $T\setminus (c_1 \cup c_2\cup c_3 \cup \iota(c_{1,2}))$ that is not $c_2$, that is, $\phi(a_k)=a_k$. Thus, $A'$ is fixed and $F'_R \cup A'$ is  a finite rigid set  with trivial pointwise stabilizer.
\end{proof}

We  define $F'_{1} := F'_R \cup A'$.

\subsection{Constructing the exhaustion for punctured surfaces}

The goal of this section is to construct an exhaustion of $\nonsep(S)$ by finite rigid sets with trivial pointwise stabilizers. In this direction, we will consider a set of generators of $\mcg(S)$ such that the subcomplex $F_1'$ satisfies hypotheses of Lemma \ref{lemma:create_exh}. We are going to assume $S=S_{g,n}$ with $n>0$ punctures.

Let $\iota$ be the back-front orientation reversing involution. Consider the usual Humphries generators \[ H' = \{\delta_{p_1}, \dots, \delta_{p_{g-1}}, \delta_{p_2^-}\} \cup \{\delta_{p_{2,i}^+}|\; 0\leq i \leq n\} \}\cup \{\delta_{c_1}, \dots, \delta_{c_k}\}\] and the half twists \[ \{h_{(k,k+1)} | 1\leq k \leq n-1\},\] where $h_{(k,k+1)}$ is the half twist that permutes the puncture $k$ with the puncture $k+1$. It is well known that  $H'\cup \{\iota\} \cup \{h_{(k,k+1)} | 1\leq k \leq n-1\}$ generates $\mcg(S)$ (see \cite[Chapter 4.4.4]{farb_primer_2012}).

\begin{lemma}\label{lemma:last}
  For every $h\in H'\cup \{\iota\} \cup \{h_{(k,k+1)} | 1\leq k \leq n-1\}$, the set $F'_{1}\cup h ( F'_{1})$ is finite rigid with trivial pointwise stabilizer.
\end{lemma}
\begin{proof}
  The proof is analogous to the closed case (see Lemma \ref{lemma:fr1_as_in_lemm}) and works directly for $h=\iota$, $h = \delta_{p_j}$, $h=\delta_{c_j}$ and $h =\delta{ {p_2}^-}$.  We consider here the rest of the cases:

  Recall that given $h=\delta_x$ and $\phi: F'_{1}\cup h(F'_{1}) \to \nonsep(S)$, we can assume that $\phi$ fixes $F'_{1}$ and $\phi|_{h(F_{1})}=\delta_x^k|_{h(F_{1})}$. Thus, the proof is a matter of checking that $k=0$. 
  
  For $h=\delta_{p_{2,j}^+}$ with $j\leq n-1$, we can consider the curves $l_j, a_j$ and plug them into  Equation (\ref{eq:intersect_0}). It follows that this is satisfied if and only if $k=0$. For $h= \delta_{p_{2,n}^+}$, one uses the curves $r_n, a_n$ in the same equation.

  Th last case to prove is $h=h_{(k-1,k)}$, which can be proved using the curves $a_{k},r_k$ in Equation (\ref{eq:intersect_0}).
\end{proof}

We now complete the proof of Theorem \ref{thm:exh}.

\begin{proof}[Proof of Theorem \ref{thm:exh} for punctured surfaces]
  The Lemmas \ref{lemma:almostlast} and \ref{lemma:last} ensure that $F'_{1}$ satisfies the hypothesis of Lemma \ref{lemma:create_exh}, which in turn produces the desired exhaustion.
\end{proof}

% Print bibliography
\printbibliography
\nocite{*}

\end{document}